\documentclass[11pt]{amsart}

\usepackage{verbatim, amssymb,hyperref}

\usepackage{color}
\usepackage{bbm}
\usepackage{graphicx}
\usepackage{subcaption}

\newcommand{\eps}{\varepsilon}
\newcommand{\dist}{\operatorname{dist}}

\newcommand{\einschraenkung}{\,\rule[-5pt]{0.4pt}{12pt}\,{}}

\newcommand{\cM}{\mathcal M}

\newcommand{\grad}{\nabla}

\newcommand{\sfrac}[2]{\mbox{$\frac{#1}{#2}$}}

\newcommand{\Id}{\operatorname{Id}}
\newcommand{\diag}{\operatorname{diag}}
\newcommand{\spec}{\operatorname{spec}}

\newcommand{\N}{{\mathbb N}}

\newcommand{\R}{{\mathbb R}}

\newcommand{\proj}{\operatorname{proj}}

\newcommand{\argmin}{\operatorname{argmin}}

\theoremstyle{plain}

\newtheorem{theorem}{Theorem}[section]

\newtheorem{lemma}[theorem]{Lemma}

\newtheorem{definition}[theorem]{Definition}

\theoremstyle{definition}
\newtheorem{remark}[theorem]{Remark}

\title[Convergence Rates for Polyak's Heavy Ball under PL-inequality]%
{Polyak's Heavy Ball Method Achieves Accelerated Local Rate of Convergence under Polyak-\L ojasiewicz Inequality}

\author[]%[Kassing]
{Sebastian Kassing}
\address{Sebastian Kassing\\
	Department of Mathematics and Informatics\\
	University of Wuppertal\\
	Germany
	}
\email{kassing@uni-wuppertal.de}

\author[]%[Weissmann]
{Simon Weissmann}
\address{Simon Weissmann\\
	Mathematisches Institut \\
	Universit\"at Mannheim \\
	Germany}
\email{simon.weissmann@uni-mannheim.de}


\begin{document}

\begin{abstract} 
	In this work, we analyze the convergence of Polyak's heavy ball method in both continuous and discrete time for non-convex $C^4$-objective functions satisfying the Polyak-\L ojasiewicz inequality.
	Under this weak assumption, we recover the asymptotic convergence rates originally derived by Polyak in~\cite{polyak1964some} for strongly convex objectives.
	Our results demonstrate that the heavy ball method exhibits asymptotic local acceleration on this class of functions. In particular, in the discrete time setting, we prove local convergence of the iterates to a minimum once the method enters a sufficiently small neighborhood of the set of minima, for a broad range of hyperparameters, including aggressive choices for the momentum parameter and the step-size for which global convergence is known to fail.
	Instead of the usually employed Lyapunov-type arguments, our approach leverages a new differential geometric perspective of the Polyak-\L ojasiewicz inequality proposed in~\cite{rebjock2023fast}. 
\end{abstract}

\maketitle

\section{Introduction}
In this work, we study the unconstrained minimization problem
\begin{align*}
	\text{Find } x^\ast \in \argmin_{x \in \R^d} f(x),
\end{align*}
where $f: \R^d \to \R$ is a continuously differentiable objective function. 
In machine learning applications, first-order optimization methods, such as the gradient descent scheme, play a central role in solving this minimization problem. Among these methods, momentum-based approaches have gained significant attention due to their potential to accelerate convergence toward critical points. The so-called momentum term stores and uses past gradient information to update the iterates. 
Motivated by the empirical success of these algorithms, a long-standing conjecture is that adding momentum improves the speed of convergence; see, e.g.~\cite{sutskever2013importance}. Heuristically, stored gradient information helps smooth out oscillations of standard gradient descent schemes, leading to improved convergence on ill-conditioned optimization problems.
However, the theoretical understanding of the advantages of these methods remains incomplete.

In this article, we consider Polyak's heavy ball method and its continuous time counterpart introduced by Polyak in his seminal work~\cite{polyak1964some}. 
The iterative scheme $(x_n)_{n \in \N_0}$, where $\N_0 := \N\cup \{0\}$, is 
defined by
\begin{align} \label{eq:HBdiscreteintro}
	x_{n+1} = x_n - \gamma \nabla f(x_n) + \beta (x_n-x_{n-1}), \quad n \in \N\,,
\end{align}
with initialization $x_0,x_1 \in \R^d$, step-size $\gamma >0$ and momentum parameter $\beta>0$. 
The algorithm \eqref{eq:HBdiscreteintro} can be viewed as a discretization of the heavy ball ODE
\begin{align}\begin{split} \label{eq:HBODE}
		\dot x_t & = v_t ,\\
		\dot v_t & = - \alpha v_t - \grad f(x_t),
	\end{split}
\end{align}
for a friction parameter $\alpha >0$.

It is worth noting that other acceleration methods, such as Nesterov’s accelerated gradient method \cite{nesterov2018lectures}, can also be interpreted as discretizations of \eqref{eq:HBODE} with asymptotically vanishing friction parameter; see, for instance, \cite{Su2016Adifferentialequation,Wilson2021}.

The first convergence rate result for the the heavy ball method \eqref{eq:HBdiscreteintro} and \eqref{eq:HBODE} was given by Polyak in the original work~\cite{polyak1964some}. It concerns the optimization of $C^2$-functions that are $\mu$-strongly convex and have an $L$-Lipschitz continuous gradient.\footnote{Polyak's result is often cited as a global convergence rate for quadratic objective functions. In addition, he established a local convergence rate under $\mu$-strong convexity.} In the continuous time setting, for optimally tuned friction parameter $\alpha$, Theorem~9 in~\cite{polyak1964some} shows that 
\begin{align} \label{eq:Polyakrate}
	\limsup\limits_{t \to \infty} e^{(2 \sqrt \mu-\eps) t} (f(x_t)-f(x^\ast)) = 0,\quad \text{for all}\ \eps>0\,,
\end{align}
where $(x_t)_{t\ge0}$ solves \eqref{eq:HBODE} and $x^\ast$ denotes the global minimum of $f$.
Moreover, the same theorem states that in the discrete time setting, for optimally tuned hyperparameters $\gamma$ and $\beta$, and under the additional assumption that heavy ball is initialized sufficiently close to the global minimum $x^\ast$, one has
\begin{equation} \label{eq:Polyakratediscrete}
	\limsup_{n \to \infty} \Bigl(\frac{\sqrt{\kappa}-1}{\sqrt{\kappa}+1}+\eps\Bigr)^{{-2n}} (f(x_n)-f(x^\ast))=0,\quad\text{for all}\ \eps>0\,,
\end{equation}
where $(x_n)_{n\in\N_0}$ is generated by \eqref{eq:HBdiscreteintro} and $\kappa = \frac{L}{\mu}$ denotes the condition number of the optimization problem. For comparison, the gradient descent scheme $(z_n)_{n \in \N_0}$, defined by
\begin{align*}
	z_{n+1} = z_n - \gamma \grad f(z_n) \quad (n \in \N_0),
\end{align*}
achieves, for optimally chosen step-size $\gamma >0$,
\begin{align*}
	{
		\limsup_{n \to \infty} 
		\Bigl(\frac{\kappa-1}{\kappa+1} +\eps  \Bigr)^{-2n} (f(z_n)-f(x^\ast)) =0,\quad\text{for all}\ \eps>0\,,}
\end{align*}
see, for example, Theorem 2.1.15 in~\cite{nesterov2018lectures}. Since $\kappa\ge1$ and, thus, $\frac{\sqrt{\kappa}-1}{\sqrt{\kappa}+1}  \le \frac{\kappa-1}{\kappa+1}$, Polyak's heavy ball method accelerates the local rate of convergence for strongly convex objectives.

In many practical applications, classical convexity assumptions are violated~\cite{vidal2017mathematics}. For example, objective functions arising in supervised learning with neural networks often exhibit complicated, non-isolated sets of critical points; see for instance~\cite{cooper2018loss, fehrman2020convergence, dereich2022minimal, wojtowytsch2021stochastic}. This has motivated decades of research on the asymptotic behavior of inertial dynamics like the solution to the heavy ball ODE \eqref{eq:HBODE}, as well as on Polyak’s heavy ball method \eqref{eq:HBdiscreteintro}, under assumptions weaker than global convexity. We will summarize related results in Section~\ref{sec:literature}. Popular relaxations of $\mu$-strong convexity include local versions of the $\mu$-quadratic growth condition, metric subregularity and quasi-strong convexity. We review these notions and their relations in Section~\ref{sec:geometryofPL}.

\begin{table} [t!]
	\setlength{\tabcolsep}{5.5pt} 
	\begin{tabular}{|c|c|c|c|}
		\hline
		Reference &
		Assumptions on $f$  & Optimal value of $\alpha$ & \begin{tabular}{c}
			Exp. rate of \\
			$f(x(t))-f(x^{\ast})$  \\
		\end{tabular} \\
		\hline
		\cite{polyak1964some} Thm. 9 & $\mu$-strongly convex &  $2 \sqrt{\mu}$ & \mathversion{bold} $2 \sqrt{\mu}$ \\
		\hline
		\cite{siegel2019accelerated} Thm. 1 &  $\mu$-strongly convex & $2 \sqrt{\mu}$ & $\sqrt{\mu}$ \\
		\hline
		\cite{aujol2022convergence1} Thm. 3.2 &  $\eqref{eq:quasar}$ + cv. + u.m. & $3 \sqrt{\frac{\mu}{2}}$ & $\sqrt{2 \mu}$  \\
		\hline
		\cite{aujol2022convergence} Thm. 1 & \eqref{eq:growth} + cv. + u.m. & $(2-\frac{\sqrt{2}}{2}) \sqrt{\mu}$ & $(2-\sqrt{2}) \sqrt{\mu}$  \\
		\hline
		\cite{aujol2024heavy} Thm. 3 &  \eqref{eq:growth} + cv. + ($\star$) & $(2-\frac{\sqrt{2}}{2}) \sqrt{\mu}$ & $(2-\sqrt{2}) \sqrt{\mu}$  \\
		\hline
		\cite{apidopoulos2022convergence} Thm. 1 &  \eqref{eq:PL} + $\kappa \le \frac 98$ & $\frac{\sqrt{\mu}}{2\sqrt{2}}(5+\sqrt{9-8 \kappa})$ & $\sqrt{2 \mu}$ \\
		\hline
		\cite{apidopoulos2022convergence} Thm. 1 & \eqref{eq:PL} + $\kappa \ge \frac 98$ & $(2\sqrt \kappa - \sqrt{\kappa-1})\sqrt{\mu})$ & $2(\sqrt \kappa -\sqrt{\kappa-1})\sqrt \mu$ \\
		\hline
		\cite{apidopoulos2022convergence} Thm. 2 &\eqref{eq:PL} + cv.  & $(2\sqrt \kappa - \sqrt{\kappa-1})\sqrt{\mu})$ & $2(\sqrt \kappa -\sqrt{\kappa-1})\sqrt \mu$ \\
		\hline 
		\cite{guptanesterov} Thm. 6 & \eqref{eq:quasar} + $C^2$-mfld.  & $2 \sqrt \mu$ & $\sqrt \mu$ \\
		\hline
		Thm.~\ref{theo1} & \eqref{eq:PL} & $2\sqrt{\mu}$ & \mathversion{bold}$2\sqrt{\mu}$
		\\
		\hline
	\end{tabular}
	\caption{Convergence rates for the heavy ball ODE \eqref{eq:HBODE}, in the sense of~\eqref{eq:Polyakrate}, obtained in the literature. The results in \cite{polyak1964some} and \cite{apidopoulos2022convergence} as well as our own result assume that $f$ is $L$-smooth. In \cite{polyak1964some} the function $f$ is required to be of class $C^2$, our result requires $f \in C^4$, while all other results allow $f \in C^1$. The abbreviations are as follows: \emph{cv.} stands for convex, \emph{u.m.} means that the objective function $f$ has a unique minimum, $C^2$-mfld. means that the set of minima forms a $C^2$-manifold, and $(\star)$ refers to the assumption that the set of minima has a $C^2$ boundary or is polyhedral. We use $\kappa := \sfrac{L}{\mu}$.}
	\label{table1}
\end{table}

\subsection{Main results}
In prior work, weakening the assumptions on the objective function led to a deteriorated rate of convergence for the heavy ball method, both in continuous and in discrete time.
The aim of this work is to prove that it is possible to recover Polyak's original asymptotic convergence rates \cite{polyak1964some} under a much weaker assumption on the geometry of the objective function. We analyze both the continuous time formulation \eqref{eq:HBODE} and the discrete time analog \eqref{eq:HBdiscreteintro} for a possibly non-convex objective function $f:\R^d \to \R$ that locally satisfies the PL-inequality \eqref{eq:PL} for a constant $\mu >0$. Our convergence analysis covers both the optimality gap in the loss function $f(x)-f(x^\ast)$, for a local minimum $x^\ast$, and the distance to the set of local minima.

Our theoretical results will hold locally around so-called regular points of $f$.
	\begin{definition}\label{ass:PL}
		Let $f: \R^d \to \R$ be a continuously differentiable function. A critical point $x^\ast \in \R^d$ is called \emph{$(\mu,L)$-regular point of $f$} if there exists an open neighborhood $U \subset \R^d$ of $x^\ast$ such that $f\big|_U:U \to \R$
		\begin{enumerate}
			\item[(A1)] is four times continuously differentiable,
			\item[(A2)] satisfies the PL-inequality for a constant $\mu >0$, i.e., 
			\begin{align} \label{eq:PL}
				\|\nabla f(x)\|^2 \ge 2 \mu (f(x)-f(x^\ast)), \quad \text{for all}\ x \in U.\tag{PL}
			\end{align}
			\item[(A3)] is $L$-smooth for a constant $L>0$, i.e.,
			\[ \|\nabla f(x)-\nabla f(y)\|\le L\|x-y\|,\quad\text{for all}\ x,y\in U\,.\]
		\end{enumerate}	
	\end{definition}

First, we consider the heavy ball ODE \eqref{eq:HBODE}. 
Table~\ref{table1} gives an overview of the convergence rates obtained in the literature for the heavy ball ODE under varying assumptions on the objective function. We stress that, except for Polyak's result \cite{polyak1964some} and the result presented in this work, all convergence rates reported in Table~\ref{table1} are non-asymptotic, and the corresponding results include explicitly derived constants.
To the best of our knowledge, the optimal\footnote{By optimal we mean that there exists a $\mu$-strongly convex and $L$-smooth quadratic function such that heavy ball attains the exponential rate of convergence $2\sqrt \mu$.} convergence rate $2 \sqrt{\mu}$ was only recovered in the very restrictive situation of a convex $C^{1,1}$-objective function that satisfies the PL-inequality \eqref{eq:PL} for a constant $\mu >0$ which is identical to the Lipschitz constant $L$ of the gradient; see~\cite[Theorem 2]{apidopoulos2022convergence}. As shown in~\cite{aujol2022convergence1}, some degree of smoothness beyond $C^1$-regularity is necessary to obtain the rate $2\sqrt{\mu}$. In fact, Proposition~3.3 in~\cite{aujol2022convergence1} presents an example of a strongly convex function that is continuously differentiable but not $L$-smooth such that the heavy ball ODE converges at the strictly slower exponential rate $\sqrt{2 \mu}$.

In our first main result, we show that for an $L$-smooth objective function that is $C^4$, one can weaken the strong convexity assumption to validity of the PL-inequality \eqref{eq:PL} and still recover Polyak's convergence rates for the heavy ball ODE.
\begin{theorem} \label{theo1}
	Let $f:\R^d \to \R$ be continuously differentiable and $x^\ast \in \R^d$ be a $(\mu,L)$-regular point. Let $(x_t,v_t)_{t\ge0}$ be a solution of \eqref{eq:HBODE} with initial condition $x_0, v_0\in\R^d$ and friction parameter $\alpha >0$. If $(x_t,v_t) \to (x^\ast,0)$ as $t\to\infty$ then, for every $\eps>0$, one has
	\begin{align} \label{eq:convrate2}
		\lim\limits_{t \to \infty} e^{(m(\alpha) - \eps) t} 
		\|x_t-x^\ast\| = 0\, ,
	\end{align}
	and
	\begin{align} \label{eq:convrate}
		\lim\limits_{t \to \infty} e^{2(m(\alpha) - \eps) t} (f(x_t)-f(x^\ast)) = 0\, ,
	\end{align}
	where $m(\alpha)=\frac12(\alpha-\sqrt{\max(0,\alpha^2-4\mu)})$. In particular, $m(\alpha)$ is maximized by $\alpha^\ast=2\sqrt{\mu}$ for which $m(\alpha^\ast)=\sqrt{\mu}$.
\end{theorem}

In our theorem, we restrict attention to the case in which the heavy ball ODE converges to a critical point which satisfies the conditions in Definition~\ref{ass:PL}. General convergence results for the ODE \eqref{eq:HBODE} are already available in the literature, including convergence for analytic objective functions \cite{haraux1998convergence}, for functions that satisfy a \L ojasiewicz inequality \cite{chill2009applications}, and for functions that satisfy a Kurdyka–\L ojasiewicz inequality around every point \cite{begout2015damped}.

We now turn to Polyak's heavy ball method in discrete time. Since we work with a constant and, thus, non-vanishing step-size $\gamma >0$ there exists no satisfying a-priori result for the convergence of $(x_n)_{n \in \N_0}$. \cite{danilova2020non} proves convergence for the heavy ball method under the PL-inequality but only for a small range of hyperparameters $\gamma, \beta$. Optimizing over these hyperparameters, the resulting rate matches the gradient descent convergence rate derived in \cite{karimi2016linear} under the same assumptions on the objective function; see also~\cite{ghadimi2015global}. In particular, no accelerated rate is obtained. For a more aggressive choice of hyperparameters, strongly convex objective functions have been constructed for which there exist initializations such that the heavy ball method does not converge; see \cite{lessard2016analysis, goujaud2023provable}. 

In more detail, in~\cite{goujaud2023provable} it is shown that for every choice of $\gamma$ and $\beta$ there exists an $L$-smooth and $\mu$-strongly convex function and initial points $x_0,x_1$ such that 
\begin{align*} 
	\limsup_{n \to \infty} \Bigl(\frac{17 \kappa-1}{17 \kappa+1} \Bigr)^{-2n} (f(x_n)-f(x^\ast)) > 0\,,
\end{align*}
implying that either the sequence $(f(x_n)-f(x^\ast))_{n\in\N_0}$ converges  significantly slower than the local rate \eqref{eq:Polyakratediscrete} derived by Polyak or the iterates $(x_n)_{n \in \N_0}$ oscillate. 
More generally, a construction of a bad instance in \cite{yue2023lower} shows that no first-order algorithm can achieve a \textit{global} accelerated rate of convergence over the class of functions satisfying the PL-inequality. 

In our second main result, we derive a \textit{local} convergence result for the heavy ball method \eqref{eq:HBdiscreteintro} under the additional assumption that the iterates find a neighborhood of a $(\mu,L)$-regular point $x^\ast$ and, in particular, that $f$ is $C^4$ on this neighborhood.
We recover the local convergence rates derived by Polyak \cite{polyak1964some} for $\mu$-strongly convex functions, proving that Polyak's heavy ball achieves asymptotic acceleration under the PL-inequality, provided the iterates find a certain neighborhood of the local minima where $f$ is $C^4$. Due to the counterexamples in~\cite{lessard2016analysis, goujaud2023provable} this additional assumption is necessary to derive an accelerated rate of convergence.

\begin{theorem} \label{theo2}
	Let $f: \R^d \to \R$ be a continuously differentiable function and $x^\ast \in \R^d$ be a $(\mu,L)$-regular point. Let $(x_n)_{n\in\N_0}$ be generated by \eqref{eq:HBdiscreteintro} for $\beta\in(0,1)$ and $\gamma\in(0,\frac{2(1+\beta)}{L})$. 
	Then there exists an open neighborhood $V$ of $x^\ast$ such that the following statement is true: if there exists an $N \in \N_0$ with $x_N,x_{N+1} \in V$ then $(x_n)_{n \in \N_0}$ converges to a local minimum $x_\infty$ with $f(x_\infty) = f(x^\ast)$ and for every $\eps >0$ one has
	\begin{align*}
		\lim_{n \to \infty} (m(\gamma, \beta)+\eps)^{-n}  \|x_n-x_\infty\| =0\, ,
	\end{align*}
	and
	\begin{align*}
		\lim_{n \to \infty} (m(\gamma, \beta)+\eps)^{-2n} (f(x_n)-f(x_\infty))=0\, ,
	\end{align*}
	where
	\begin{align*}
		m&(\gamma, \beta) = \\
		&\begin{cases}
			\sqrt \beta , & \text{ if } \gamma \in \mathcal I_{\beta,\mu, L}, \\
			\frac{1+\beta-\gamma \mu}{2}+\sqrt{\left(\frac{1+\beta-\gamma \mu}{2}\right)^{2}-\beta}, & \text{ if }  \gamma \in \Bigl(0, \frac{2(1+\beta)}{L+\mu} \Bigr] \setminus \mathcal I_{\beta,\mu, L} , \\
			\frac{\gamma L-(1+\beta)}{2}+\sqrt{\left(\frac{\gamma L-(1+\beta)}{2}\right)^{2}-\beta}, & \text{ if } \gamma \in \Bigl( \frac{2(1+\beta)}{L+\mu} , \frac{2(1+\beta)}{L} \Bigr) \setminus \mathcal I_{\beta,\mu, L}.\\
		\end{cases}
	\end{align*} 
	with $\mathcal I_{\beta,\mu, L}:= \Big[\frac{(1-\sqrt{\beta})^2}{\mu}, \frac{(1+\sqrt{\beta})^2}{L}\Big]$.
	This term is minimized by choosing \begin{align*}
		\gamma=\frac{4}{(\sqrt{\mu}+\sqrt{L})^{2}} \quad \text { and } \quad \beta= \Bigl( \frac{\sqrt \kappa-1}{\sqrt{\kappa}+1}\Bigr)^2
	\end{align*}
	for which $m(\gamma, \beta) = \frac{\sqrt{\kappa}-1}{\sqrt{\kappa}+1}$.
\end{theorem}

We note that for $\beta \in \Bigl(0, \Bigl( \frac{\sqrt \kappa-1}{\sqrt{\kappa}+1}\Bigr)^2\Bigr)$ the interval $\mathcal I_{\beta,\mu, L}$ denotes the empty set. Moreover, for $\beta = \Bigl( \frac{\sqrt \kappa-1}{\sqrt{\kappa}+1}\Bigr)^2$ one has $\frac{(1-\sqrt{\beta})^2}{\mu}= \frac{(1+\sqrt{\beta})^2}{L}= \frac{2(1+\beta)}{L+\mu} $.

The derivation of convergence rates typically relies on crafting a suitable Lyapunov function depending on the class of objective functions under consideration. The proof strategy then consists of deriving a convergence rate for the Lyapunov function and, in a second step, proving that this implies a convergence rate for the objective function value. In contrast to this, our proof exploits a new differential geometric perspective on the PL-inequality; see \cite{feehan2020morse, rebjock2023fast}. More specifically, if $f$ is $C^2$ and satisfies the PL-inequality \eqref{eq:PL}, then the set of minima forms a manifold, and the function grows quadratically with the distance to this manifold. In that sense, we consider the dynamics of the heavy ball method under a chart that separates the space into tangential and normal directions. In the normal directions the manifold is strongly attracting, whereas in the tangential directions there is no restoring force. First, we prove fast convergence in the normal directions. This implies that the gradient of the objective function vanishes exponentially fast, which, due to the friction in the dynamical system, then leads to exponential convergence in the tangential directions as well. Therefore, the distance traveled by heavy ball can be bounded, which in turn leads to convergence of the dynamical system. A similar geometric approach has been used in \cite{guptanesterov} to show fast convergence rates of Nesterov's accelerated gradient descent in a non-convex setting.

The remainder of this article is organized as follows. Section~\ref{sec:literature} reviews related results from the literature. Section~\ref{sec:geometryofPL} discusses the geometric implications of the PL-inequality and its relationship to other relaxations of strong convexity. The proofs of Theorems~\ref{theo1} and \ref{theo2} are presented in Sections~\ref{sec:continuous} and Section~\ref{sec:discrete}, respectively. In Section~\ref{sec:experiments} we illustrate our findings through a numerical toy example.

Throughout this article, we slightly abuse notation and identify the spaces $\R^{d_1}\times \R^{d_2}$ and $\R^{d_1+d_2}$ for $d_1,d_2 \in \N$. In more detail, for vectors $v_1 \in \R^{d_1}$ and $v_2 \in \R^{d_2}$, we identify the object $(v_1, v_2) \in \R^{d_1} \times \R^{d_2}$ with $v \in \R^{d_1+d_2}$ such that $v^{(i)}=v_1^{(i)}$ for all $i=1, \dots, d_1$ and $v^{(d_1+i)}=v_2^{(i)}$ for all $i=1, \dots, d_2$.

\section{Literature review}\label{sec:literature}
Early work on second-order gradient systems was carried out by Haraux and Jendoubi in \cite{haraux1998convergence}. They established convergence to equilibrium for dynamical systems of the form
	\[ \ddot x_t + g(x_t) = \nabla f(x_t) \]
	where $f$ is analytic and $g$ is Lipschitz and coercive. Their framework includes the heavy ball ODE as a special case and represents one of the first general convergence results for inertial optimization dynamics beyond the convex setting. In a related direction, \cite{doi:10.1142/S0219199700000025} studied the heavy ball ODE in Hilbert spaces for smooth functions that are either convex or Morse. Therein, global existence of trajectories was established, weak convergence was obtained under convexity, and asymptotic convergence to local minima was shown in the Morse case, while examples of divergent solutions were also constructed. Subsequent works investigated second-order gradient systems with asymptotically vanishing viscous damping, motivated by fast convergence in convex optimization. In particular, \cite{GEG2009,Cabot2009,10.1007/s10107-016-0992-8} analyzed long-time behavior and convergence properties of inertial dynamics with time-dependent damping coefficients. It was observed that, for a specific choice of vanishing damping, these second-order systems can be interpreted as continuous-time limits of Nesterov’s accelerated gradient method \cite{Su2016Adifferentialequation}. A refined class of such models, incorporating Hessian-driven damping, was studied in \cite{ATTOUCH20165734}, where well-posedness and fast convergence were established in the convex regime. In \cite{maulen2025sde}, the stability of these dynamical systems under stochastic perturbation has been analyzed.

A broad line of research investigates the heavy ball ODE and heavy ball methods under non-convex regularity assumptions; see e.g.~\cite{siegel2019accelerated, aujol2022convergence1, aujol2022convergence, apidopoulos2022convergence, aujol2024heavy}. Convergence of discrete heavy ball schemes or stochastic variants has been established under various structural assumptions, including weak quasi-convexity or metric subregularity conditions~\cite{danilova2020non, wang2022provable, gess2023convergence, dereich2021convergence, weissmann2024almost}. Further, \cite{goudou2009gradient} studied gradient flow, heavy ball dynamics and proximal inertial variants under quasi-convexity in Hilbert spaces, establishing weak (and in symmetric cases strong) convergence. The work extends the convex theory presented in \cite{doi:10.1137/S0363012998335802}. More recent developments have also examined inertial Newton-like systems in Morse settings~\cite{castera_inertial_2023}, which include the heavy ball ODE as a limiting case and show almost-sure avoidance of strict saddle points. A central theme in all these works are geometric assumptions that lie strictly between strong convexity and general non-convexity. Metric subregularity emerged originally from variational analysis~\cite{cornejo1997conditioning}, and have since been refined in the non-smooth and semi-algebraic setting through the works of \cite{drusvyatskiy2015quadratic, Drusvyatskiy2014}.

To the best of our knowledge, \cite{wang2022provable} is the only work that shows an accelerated exponential rate of convergence for the heavy ball method beyond strongly convex objective functions. However, the function class covered by Theorem 1 in \cite{wang2022provable} is still quite restrictive. In particular, in addition to assuming the PL-inequality, they assume that, at each iteration, the Hessian is a diagonal matrix. Moreover, they assume that there exists a global minimum $x^\ast$ such that, at each iteration $n \in \N$, when averaging the Hessian of $f$ on the straight line from $x^\ast$ to $x_n$ one gets a strictly positive definite matrix. Note that assuming this property on a neighborhood of $x^\ast$ would imply that $x^\ast$ is an isolated strongly attracting minimum. Our Theorem~\ref{theo2} shows that the assumptions in Definition~\ref{ass:PL} are sufficient to get an accelerated local convergence rate for Polyak's heavy ball method. Let us also mention \cite{wang2021modular}, where a particular matrix decomposition is used to avoid the application of Gelfand's formula and to achieve a non-asymptotic accelerated convergence rate for strongly convex quadratics.

One of the most prominent assumptions on the objective function was introduced by Polyak \cite{polyak1963gradient} and \L ojasiewicz \cite{lojasiewicz1963propriete} and is known as the Polyak-\L ojasiewicz (PL) inequality \eqref{eq:PL}. 
	Due to its simplicity, generality and applicability for objective functions appearing in machine learning, see e.g.~\cite{karimi2016linear, kuruzov2023gradient, garrigos2023square, wojtowytsch2021stochastic, dereich2021convergence}, the PL-inequality is becoming a popular assumption for proving fast convergence for optimization algorithms on non-convex objective functions. It provides a gradient-dominance property ensuring global convergence of gradient-based methods without convexity (see Definition~\ref{ass:PL}). Related Łojasiewicz-type inequalities have played a central role in non-convex optimization more broadly: the Kurdyka–\L ojasiewicz (KL) framework \cite{attouch_convergence_2013, attouch2013convergence} provides a unified analysis for a wide class of descent algorithms including proximal, forward–backward and inexact gradient methods. Important connections between the quadratic growth condition \eqref{eq:growth} and KL inequalities were developed in the works \cite{bolte2017errorbounds,aze_nonlinear_2014}.

This KL perspective underlies many results for non-smooth or composite optimization involving inertial or proximal terms. For instance, the iPiano algorithm~\cite{doi:10.1137/130942954} is a non-smooth forward-backward generalization of the heavy ball method. In \cite{doi:10.1137/130942954}, the authors prove global convergence to critical points under the KL property together with sublinear convergence of the proximal gradient residual for this algorithm. A similar KL-based convergence analysis has been developed for non-convex forward–backward splitting in~\cite{10.5555/3157382.3157549}. The same authors used differential geometric and linearization arguments (similar to the ones used in the present work), to show local linear convergence of forward-backward splitting methods; see \cite{liang2017activity}. Moreover, a local convergence analysis for the heavy ball and iPiano methods under local KL inequalities have been developed in~\cite{ochs2018local}, showing local linear convergence to local minima. In the convex composite setting, global acceleration is classically achieved by Nesterov-type methods, most prominently the Fast Iterative Shrinkage-Thresholding Algorithm (FISTA) \cite{beck2009}.

Finally, the role of KL-type and PL-type inequalities in infinite-dimensional settings relevant for inverse problems has also been explored. In \cite{garrigos2023convergence}, the authors derived a Łojasiewicz inequality for certain convex, non-smooth quadratic functionals on Hilbert spaces and proved linear convergence of forward–backward splitting.

\section{The Geometry of the PL-Inequality} \label{sec:geometryofPL}
Understanding the geometry of the loss landscape for an optimization task is an important part of mathematical programming. The performance of optimization algorithms highly depends on the geometric assumptions satisfied by the objective function $f$. Classical results on first-order methods often assume $f$ to be $\mu$-strongly convex for a $\mu >0$, which means that
\begin{align*}
	f(y) \ge f(x) + \langle \nabla f(x), (y-x) \rangle + \frac{\mu}{2} \|y-x\|^2,\quad \text{for all}\ x,y \in \R^d \,.
\end{align*}
In many applications, this assumption turns out to be too restrictive. For example, strong convexity implies that the global minimum is unique. 

Several generalizations of $\mu$-strong convexity have been suggested in the literature, aiming at weakening the assumptions on the geometry of $f$, while maintaining fast convergence of first-order methods to a minimum. Let us mention a few weaker assumptions that have been used, e.g., to derive convergence rates of the heavy ball ODE \eqref{eq:HBODE}. We fix a level $\ell \in \R$ and define the set of local minima at this level as
	\begin{align*}
		\mathcal N = \{x \in \R^d: x \text{ is a local minimum of } f \text{ and } f(x)= \ell\}.
	\end{align*}
	Assume that $\mathcal N$ is non-empty and denote by $d(x,\mathcal N):= \inf_{y \in \mathcal N} \|x-y\|$ the Euclidean distance of $x \in \R^d$ to $\mathcal N$. Moreover, we define the projection of $x\in\R^d$ onto $\mathcal N$ by $\proj_{\mathcal N}(x) := \argmin_{y \in \mathcal N}\|x-y\|$, which might be empty or non-unique.

Let $x^\ast$ be a point in $\mathcal N$ and $U$ be a neighborhood of $x^\ast$. For $\mu >0$,  $f$ is said to satisfy the $\mu$-quadratic growth condition on
$U$ if 
\begin{align} \label{eq:growth} \tag{QG}
	f(x)-f(x^\ast) \ge \frac{\mu}{2} d(x,\mathcal N)^2,\quad \text{for all}\ x \in U\,,
\end{align}
the metric subregularity condition with modulus $1/\mu$ on $U$ if
\begin{align} \label{eq:error} \tag{MS}
	\mu \, d(x,\mathcal N) \le \|\nabla f(x)\|,\quad \text{for all}\ x \in U\,,
\end{align}
and the $\mu$-quasi-strong convexity condition on $U$ if for all $x \in U$ and $y\in \proj_{\mathcal N}(x) \cap U$
\begin{align} \label{eq:quasar} \tag{QSC}
	f(y) \ge f(x) + \langle \nabla f(x),y-x \rangle+ \frac{\mu}{2} \|x-y \|^2.
\end{align}
All of the above conditions allow $f$ to be a non-convex objective function. For the concept of metric (sub-)regularity of sub-differentials we refer to \cite{AragonArtacho2008Characterization}.

It turns out that these assumptions, as well as the PL-inequality \eqref{eq:PL}, are intimately related to each other, see, e.g. \cite{karimi2016linear, rebjock2023fast, necoara2019linear, apidopoulos2022convergence}. Most importantly, when $f$ is $C^2$, then the conditions $\eqref{eq:PL}$, \eqref{eq:growth}, \eqref{eq:error} and \eqref{eq:quasar} are at least locally equivalent in the following sense: if one of the above equations is satisfied for $\mu >0$ on a neighborhood $U$ of a local minimum $x^\ast \in \mathcal N$, then for every $\mu'<\mu$ there exists a possibly smaller neighborhood $U' \subset U$ of $x^\ast$ such that all of the conditions are satisfied on $U'$ for the constant $\mu'$.
This has strong implications for the local analysis of optimization algorithms. In fact, when $f$ is sufficiently regular (namely $f \in C^2$) one can conduct a universal analysis assuming one (and thus all) of the above conditions. 
Therefore, Polyak's heavy ball method also achieves an accelerated local rate of convergence in the sense of Theorem~\ref{theo1} and Theorem~\ref{theo2} when the PL-inequality \eqref{eq:PL} is replaced by \eqref{eq:growth}, \eqref{eq:error} or \eqref{eq:quasar}. Consequently, for $f \in C^4$, the optimal convergence rate for the optimality gap of the heavy ball ODE \eqref{eq:HBODE}, as well as for Polyak's heavy ball method \eqref{eq:HBdiscreteintro}, is (in all of the above cases) given by $2\sqrt{\mu}$ and $\bigl( \frac{\sqrt \kappa-1}{\sqrt \kappa+1} \bigr)^2$, respectively.

The local equivalence between the PL-inequality~\eqref{eq:PL}, the $\mu$-quadratic growth condition \eqref{eq:growth} and the metric subregularity with modulus $1/\mu$ \eqref{eq:error} was shown for example in \cite{bolte2017errorbounds,aze_nonlinear_2014,rebjock2023fast}. The central observation linking these geometric assumptions on the objective function is the following fact derived in~\cite{rebjock2023fast}, which extends an earlier result in \cite{feehan2020morse}. In fact, this result is also the starting point for our analysis for deriving an accelerated local rate of convergence for the heavy ball method.

\begin{theorem}[See Theorem~2.16 and Corollary~2.17 in~\cite{rebjock2023fast}] \label{thm:rebjock} 
		Let $x^\ast \in \mathcal N$ and $U\subset \R^d$ be a neighborhood of $x^\ast$ such that $f\big|_U$ is $C^k$, for a $k \ge 2$, and satisfies the PL-inequality \eqref{eq:PL} with constant $\mu >0$. 
		Then, there exists a neighborhood $U'$ of $x^\ast$ such that the set $\mathcal M :=\mathcal N \cap U'$ 
		is a $C^{k-1}$-submanifold of $\R^d$ and the Hessian $\nabla^2 f$ satisfies
		\begin{align} \label{eq:MB}
			\langle v, \nabla^2 f(x) \,  v \rangle \ge \mu \|v\|^2 \quad \text{ for all } x \in \mathcal M, v \in N_x \mathcal M.
		\end{align}
		Here, $N_x \mathcal M$ denotes the normal space of $\mathcal M$ at $x$, i.e. $N_x \mathcal M = (T_x \mathcal M )^\perp$, where $T_x \mathcal M$ denotes the tangent space of $\mathcal M$ at $x$. 
\end{theorem}

Note that Theorem~\ref{thm:rebjock} implies that $\ker (\nabla^2 f(x))=T_x \mathcal M$ 
for all $x \in \mathcal M$, where $\mathcal M$ denotes the manifold given by Theorem~\ref{thm:rebjock}; see \cite[Corollary~2.17]{rebjock2023fast} for more details.

In \cite[Proposition 1]{apidopoulos2022convergence} it is shown that $\mu$-quasi strong convexity on $\R^d$ implies the PL-inequality for the constant $\mu$. In the following result, we use Theorem~\ref{thm:rebjock} to prove that, at least locally, the converse is also true. 
Other relations for quasi-strong convexity have only been shown in the convex setting \cite{necoara2019linear, aujol2022convergence1}. 

\begin{lemma} Suppose that $f$ is $C^2$. Let $x^\ast \in \mathcal N$ be a local minimum and $U$ be a neighborhood of $x^\ast$. 
	If $f$ satisfies the PL-inequality with parameter $\mu$ on $U$, then for all $\mu' < \mu$ there exists a neighborhood $U' \subset U$ of $x^\ast$ such that $f$ is $\mu'$-quasi-strongly convex on $U'$.
\end{lemma}

\begin{proof} 
		By Theorem~\ref{thm:rebjock} above, there exists a neighborhood $V$ of $x^\ast$ such that $\mathcal M:=\mathcal N \cap V$ is a $C^1$-submanifold of $\R^d$.
		Using Lemma~1.5 in \cite{rebjock2023fast}, there exists a convex and compact neighborhood $K \subset V$ of $x^\ast$ such that for all $x \in K$ one has $\proj_{\mathcal N}(x)\neq \emptyset$ and for all $y \in \proj_{\mathcal N}(x)$, one has $y \in \mathcal M$ and $x-y \in N_y \mathcal M$.
		Fix a $\mu' \in (0,\mu)$ and set $\eps:= \mu-\mu'>0$. Since $f \in C^2$, the Hessian $\nabla^2 f$ is uniformly continuous on $K$. Hence, there exists a $\rho>0$ such that for all $x,y\in K$ with $\|x-y\|\le \rho$ one has $\|\nabla^2 f(x)-\nabla^2 f(y)\|\le \eps$. Let $U':= \{x \in K: \dist(x,\mathcal M)\le \rho\}$.
		
		We show that $f$ is $\mu'$-quasi strongly convex on $U'$. Let $x,y\in U'$ with $y \in \proj_{\mathcal N}(x)$. Then, $y \in \mathcal M \cap K$ and $\|x-y\|\le \rho$. For $t \in [0,1]$ set $z_t:= x+t(y-x)$ and note that $\|z_t-y\| \le \rho$ such that $\|\nabla^2 f(z_t)-\nabla^2 f(y)\|\le \eps$. Thus,
		\begin{allowdisplaybreaks}
			\begin{align*}
				f(y)-f(x) &= \int_0^1 \langle \nabla f(z_t),y-x\rangle \, dt \\
				&=  \langle \nabla f(x),y-x \rangle +  \int_0^1 \int_0^t \langle y-x,\nabla^2 f(z_s) (y-x) \rangle \, ds \, dt \\
				& \ge \langle \nabla f(x),y-x \rangle +  \frac 12 \big(\langle y-x,\nabla^2 f(y) (y-x) \rangle - \eps \|y-x\|^2 \big) \\
				& \ge \langle \nabla f(x),y-x \rangle +  \frac 12 \big(\mu \|y-x\|^2 - (\mu-\mu') \|y-x\|^2\big),
			\end{align*}
			where in the last inequality we used Theorem~\ref{thm:rebjock}.
		\end{allowdisplaybreaks}
\end{proof}

\section{Proof of Theorem~\ref{theo1}} \label{sec:continuous}
In the following section, we present the proof of Theorem~\ref{theo1}. The proof is split into several intermediate steps. Recall that $x^\ast$ is a $(\mu,L)$-regular point and by assumption $\lim_{t \to \infty} x_t=x^\ast$. For simplicity, throughout this section we assume that $f(x^\ast)=0$.

We use Theorem~\ref{thm:rebjock} to conduct a local analysis of the convergence rate in a neighborhood of $x^\ast$. In Section~\ref{sec:geometric}, we construct a coordinate chart for the normal bundle of the manifold $\mathcal M$ (given by Theorem~\ref{thm:rebjock}) that separates the tangential directions from the normal directions of $\mathcal M$. In Section~\ref{sec:linsys}, we analyze the heavy ball ODE under the coordinate chart. In Section~\ref{sec:final} we conclude our proof.

\subsection{Geometrical preparation} \label{sec:geometric}
In the first part of the proof, we employ Theorem~\ref{thm:rebjock} to construct a certain coordinate chart around the limiting stationary point $x^\ast$. This allows us to transfer the dynamical system \eqref{eq:HBODE} into a system with favorable structural properties. More specifically, we build a bundle chart for the normal tube of the manifold inspired by Fermi coordinates; see, e.g. \cite[Chapter 2]{gray2003tubes}. Additionally, we show certain Taylor expansion results for the chart and the objective function under the chart. 

Let $x^\ast$ be a $(\mu,L)$-regular point and, hence, a local minimum of $f$. 
Theorem~\ref{thm:rebjock} states that, in a neighborhood of $x^\ast$, the local minima form a $C^3$-submanifold $\mathcal M$ of $\R^d$. 
We denote $d_T:= \dim(\mathcal M)$ and $d_N:=d-d_T$. Thus, there exists an open and bounded set $\hat U \subset \R^d$ containing $x^\ast$ and a $C^3$-diffeomorphism $\hat \Phi:\hat U \to \hat \Phi(\hat U)$ such that $\hat \Phi(x^\ast)=0$ and
\begin{align*}
	\hat \Phi(\mathcal M\cap \hat U) = \hat \Phi(\hat U) \cap (\R^{d_T} \times \{0\}^{d_N}).
\end{align*}
Note that $\hat \Phi(\hat U) \cap (\R^{d_T} \times \{0\}^{d_N}) \ni (\zeta,0) \mapsto f(\hat \Phi^{-1}(\zeta,0))$ remains constant since $\hat \Phi^{-1}(\zeta,0)$ maps to a subset of $\mathcal M$  and $f$ is constant on $\mathcal M$.

After rotating the optimization problem, we may assume without loss of generality that $D \hat \Phi(x^\ast)=\Id_{\R^d}$. More precisely, let $Q$ be an orthogonal matrix satisfying 
	$$
	Q \cdot  T_{x^\ast}\mathcal M = \R^{d_T} \times \{0\}^{d_N}.
	$$
	Instead of optimizing $f$, we may equivalently consider the heavy ball scheme applied to $f \circ Q^\top$.
	In these rotated coordinates, the Jacobi matrix $D\hat \Phi(x^\ast)$ is invertible and maps $\R^{d_T} \times \{0\}^{d_N}$ onto itself. Now, we can consider $\hat \Phi_1:= (D\hat \Phi(x^\ast))^{-1} \hat \Phi$, which is a $C^3$-diffeomorphism satisfying all the above properties and, additionally, $D \hat \Phi_1 (x^\ast)=(D\hat \Phi(x^\ast))^{-1}D\hat \Phi(x^\ast)=\Id_{\R^d}$.

The next result follows immediately from \cite[Lemma~C.1]{dereich2023central}. The proof is based on the application of a Gram-Schmidt orthonormalization procedure to the column vectors of the invertible matrix $D\hat \Phi^{-1}(x)$ for each $x \in \hat \Phi(\mathcal M \cap \hat U)$. Recall that $N_x \mathcal M = (T_x \mathcal M)^\perp$ denotes the normal space of $\mathcal M$ at $x$.
\begin{lemma} \label{lem:chart} Let $\mathcal M$ be a $C^3$-submanifold of $\R^d$ and $x^\ast \in \mathcal M$ with $T_{x^\ast}\mathcal M = \R^{d_T} \times \{0\}^{d_N}$.
	Then there exists an open neighborhood $\mathcal U$ of $x^\ast$ and a $C^{2}$-diffeomorphism $\Phi: \mathcal U \to \Phi(\mathcal U)$ that satisfies:
	\begin{enumerate}
		\item[(a)] $\Phi(x^\ast)=0$, $D\Phi(x^\ast)=\Id_{\R^d}$ and $
		\Phi(\mathcal M \cap \mathcal U) = \Phi(\mathcal U) \cap (\R^{d_T} \times \{0\}^{d_N})$, 
		\item[(b)] for every $(\zeta,\theta) \in \R^{d_T} \times \R^{d_N}$ with $(\zeta, \theta) \in \Phi(\mathcal U)$ one has $(\zeta,0) \in \Phi(\mathcal U)$, and
		\item[(c)] there exists a family $(P_x:x \in \mathcal M \cap \mathcal U)$ of isometric isomorphisms $P_x:\R^{d_N} \to N_x \mathcal M$ such that for every $(\zeta, \theta) \in \Phi(\mathcal U) \subset \R^{d_T} \times \R^{d_N}$
		\[
		\Phi^{-1}(\zeta,\theta) = \Phi^{-1}(\zeta,0) + P_{\Phi^{-1}(\zeta,0)} (\theta).
		\]
	\end{enumerate}
\end{lemma}

\begin{figure}[htb!]
	\begin{center}
		\includegraphics[width=1\textwidth]{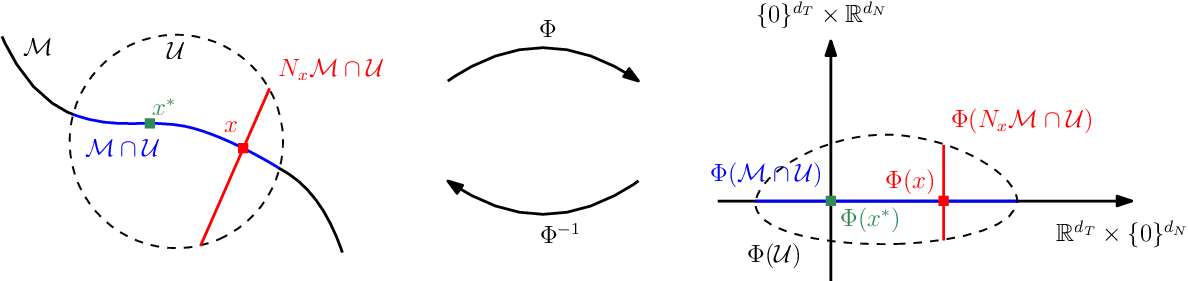}
	\end{center}
	\caption{Visual illustration of the chart $\Phi$ from Lemma~\ref{lem:chart}.}\label{fig:chart}
\end{figure}

	\begin{remark}
		We briefly explain the smoothness assumption $f \in C^4$ used in our main results. Our proof adapts Polyak's original technique \cite{polyak1964some}, which requires a $C^2$-objective, by analyzing the heavy ball method under the change of variables induced by the manifold chart of the normal bundle constructed in Lemma~\ref{lem:chart}. Accordingly, we require $f \circ \Phi^{-1} \in C^2$. Using Theorem~\ref{thm:rebjock}, if $f \in C^k$ satisfies a local PL-inequality, then the set of local minima forms a $C^{k-1}$-submanifold. However, our construction in Lemma~\ref{lem:chart} relies on a chart of the normal bundle, which is only $C^{k-2}$. Thus, assuming $f \in C^4$ ensures the existence of a chart $\Phi \in C^2$ satisfying the required normality conditions.
	\end{remark}

In the next section, we will consider the dynamical system \eqref{eq:HBODE} under the chart $\Phi$. We will show that $(\Phi(x_t))_{t \ge 0}$ can approximatively be considered as the solution to the heavy ball ODE for the objective function $\tilde f:= f \circ \Phi^{-1}$. Next, we derive properties for the chart $\Phi$ and the function $\tilde f$ that will be used in the subsequent analysis. In Figure~\ref{fig:chart} we provide a visual illustration of the map $\Phi$. The map $\Phi$ converts, isometrically, straight lines that are perpendicular to the tangent space of $\mathcal M$ at a point $x \in \mathcal M$ to straight lines that are perpendicular to $\R^{d_T} \times \{0\}^{d_N}$. Consequently, the eigenvalues of $\nabla^{2} f(x^\ast)$ coincide with those of $\nabla^{2}(f \circ \Phi^{-1})(0)$ and
	$
	\ker(\nabla^{2} (f \circ \Phi^{-1})(0))
	= \mathbb R^{d_T} \times \{0\}^{d_N}$.
We verify these statements below.

\begin{lemma} \label{lem:ftilde} The function $\tilde f:=f \circ \Phi^{-1}$ satisfies:
	\begin{enumerate}
		\item[(i)] $ \nabla \tilde f(0)  = 0  $ and $\nabla^2 \tilde f(0) \, v =0$ for all $v \in \R^{d_T}\times \{0\}^{d_N}$
		\item[(ii)] $\lambda\in\R$ is an eigenvalue of $\nabla^2 f(x^\ast)$ if and only if it is an eigenvalue of $\nabla^2(f\circ \Phi^{-1})(0)$. Moreover, one has
		\[\mu\|v\|^2 \le \langle v,\nabla^2 \tilde f(0) \,  v\rangle \le L\|v\|^2 \]
		for any $v\in \{0\}^{d_T} \times \R^{d_N}$.
	\end{enumerate} 
\end{lemma}

\begin{proof} (i):  The first claim follows immediately from the chain rule and the fact that $\nabla f(x^\ast)=0$.
	For the Hessian, we compute
	\[\langle \nabla^2\tilde f(0)\, v,v\rangle = \frac{d^2}{dt^2} \, f(\Phi^{-1}(tv)) \einschraenkung_{t=0} = 0\,,\]
	for every $v \in \R^{d_T} \times \{0\}^{d_N}$, where we have used that $\zeta \mapsto f(\Phi^{-1}(\zeta,0))$ remains constant. 
	
	(ii): For $v= (0,\tilde v) \in \{0\}^{d_T} \times \R^{d_N}$ we consider
	\begin{align*}
		\langle \nabla^2 \tilde f(0) \, v, v \rangle &= \frac{d^2}{dt^2} \, f(\Phi^{-1}(0, t \tilde v)) \einschraenkung_{t=0} \\
		&= \frac{d^2}{dt^2} \, f(x^\ast + t P_{x^\ast}(\tilde v)) \einschraenkung_{t=0} = \langle \nabla^2 f(x^\ast)\,  P_{x^\ast}(\tilde v)), P_{x^\ast}(\tilde v)) \rangle ,
	\end{align*}
	where $\|v\|=\|P_{x^\ast}(\tilde v)\|$. Since $P_{x^\ast}:\R^{d_N}\to N_{x^\ast}\mathcal M$ is bijective we find that the eigenvalues of $\nabla^2 f(x^\ast)$ are equal to the eigenvalues of $\nabla^2 \tilde f(0)$ and, in particular, using Theorem~\ref{thm:rebjock} and $L$-smoothness of $f$ around $x^\ast$ it follows that
	\begin{align*}
		\mu\|v\|^2 \le \langle v,\nabla^2 \tilde f(0) \,  v\rangle \le L\|v\|^2 ,
	\end{align*}
	for all $v \in N_{x^\ast} \mathcal M$.
\end{proof}

Next, we derive several Taylor expansion results for $\Phi$ and $\tilde f$.
Denote by $\proj_T: \R^d \to \R^{d_T} \times \{0\}^{d_N}$ the orthogonal projection onto $\R^{d_T} \times \{0\}^{d_N}$ and by $\proj_N: \R^d \to \{0\}^{d_T} \times \R^{d_N}$ the orthogonal projection onto $\{0\}^{d_T} \times \R^{d_N}$. Due to Lemma~\ref{lem:chart} one has $\proj_T(\tilde x) \in \Phi(\mathcal U)$ for all $\tilde x \in \Phi(\mathcal U)$. Further shrinking the neighborhood $\mathcal U$ leads to the following useful properties. 
\begin{lemma}\label{lem:Udelta}
	For any fixed $\delta>0$ there exists a neighborhood $\mathcal U_\delta \subset \mathcal U$ of $x^\ast$ that still satisfies properties \emph{(a)-(c)} of Lemma~\ref{lem:chart} such that for all $x \in \mathcal U_\delta$, $\tilde x \in \Phi(\mathcal U_\delta)$:
	\begin{enumerate}
		\item [(i)] $\|(D\Phi^{-1}(\Phi(x)))^\top-D\Phi(x) \|  \le \frac{\delta}{16 L}$,
		\item[(ii)] $\|\nabla \tilde f(\tilde x) - \nabla^2 \tilde f(\proj_T(\tilde x))\proj_N(\tilde x)\|\le \frac{\delta}{4} \|\proj_N(\tilde x)\|$,
		\item[(iii)] $\|\nabla^2 \tilde f(\proj_T(\tilde x))-\nabla^2 \tilde f(0))\| \le \frac{\delta}{4}$,
		\item[(iv)] $\|D \Phi(x)\| \le 2$, $\|D \Phi^{-1}(\tilde x)\| \le 2$, and $\sup_{x \in \mathcal U_\delta} \|D^2\Phi(x)\|<\infty$ and 
		\item[(v)] $\|\grad \tilde f(\tilde x)-\grad\tilde f(\proj_T(\tilde x))\|
			\le 2 L \|\proj_N(\tilde x)\|$.
	\end{enumerate}
\end{lemma}

	\begin{proof}
		Since $\Phi$ is $C^2$ on $\mathcal U$ with $\Phi(x^\ast)=0$, $D\Phi(x^\ast)=\Id_{\mathbb R^d}$ and $D\Phi^{-1}(0)=\Id_{\mathbb R^d}$, see Lemma~\ref{lem:chart},
		there exists a neighborhood $\mathcal U_\delta \subset \mathcal U$ such that for all $x \in \mathcal U_\delta$
		\[
		\|(D\Phi^{-1}(\Phi(x)))^\top-D\Phi(x) \|  \le \frac{\delta}{16 L}.
		\]
		Moreover, we may further shrink $\mathcal U_\delta$ so that for all $x \in \mathcal U_\delta$, $\tilde x \in \Phi(\mathcal U_\delta)$
		\[
		\|D\Phi(x)\| \le 2, \qquad \|D\Phi^{-1}(\tilde x)\| \le 2,
		\]
		and $\sup_{x\in\mathcal U_\delta}\|D^2\Phi(x)\|<\infty$.
		
		Next, note that $\tilde f\in C^2(\Phi(\mathcal U))$ and for all $\tilde x=\proj_N(\tilde x)+\proj_T(\tilde x)\in\Phi(\mathcal U)$ one has $\nabla \tilde f(\proj_T(\tilde x))=0$. Thus, a first-order Taylor expansion together with the continuity of $\nabla^2 \tilde f$ ensures the existence of a neighborhood $\mathcal U_\delta$ such that
		\[
		\|\nabla \tilde f(\tilde x) - \nabla^2 \tilde f(\proj_T(\tilde x))\proj_N(\tilde x)\|
		\le \frac{\delta}{4}\|\proj_N(\tilde x)\|.
		\]
		Similarly, since $\nabla^2 \tilde f$ is continuous, we may shrink $\mathcal U_\delta$ so that for all $\tilde x \in \Phi(\mathcal U_\delta)$,
		\[
		\|\nabla^2 f(\proj_T(\tilde x))-\nabla^2 f(0)\| \le \frac{\delta}{4}.
		\]
		
		Finally, using that $\|\nabla^2 \tilde f(0)\|\le L$ and the continuity of $\nabla^2 \tilde f$, we may assume that
		\[
		\|\nabla^2 \tilde f(\tilde x)\| \le 2L
		\quad\text{for all }\tilde x\in \Phi(\mathcal U_\delta).
		\]
		By the mean value theorem applied to $\nabla \tilde f$ along the segment between $\proj_T(\tilde x)$ and $\tilde x$, we obtain (after possibly shrinking $\mathcal U_\delta$)
		\[
		\|\nabla \tilde f(\tilde x)-\nabla \tilde f(\proj_T(\tilde x))\|
		\le 2L \|\proj_N(\tilde x)\|.
		\]
		Choosing $\mathcal U_\delta$ as the intersection of the neighborhoods required above completes the proof.
	\end{proof}

\subsection{Dynamical system under the coordinate chart} \label{sec:linsys}
In the following, $\mathcal M$ denotes the $C^3$-manifold that is given by Theorem~\ref{thm:rebjock}, $\mathcal U \subset \R^d$ denotes the neighborhood of $x^\ast$ and $\Phi:\mathcal U \to \Phi(\mathcal U)$ denotes the $C^2$-mapping that is given by Lemma~\ref{lem:chart} and, for all $\delta>0$, $\mathcal U_\delta$ denotes the neighborhood of $x^\ast$ that is given by Lemma~\ref{lem:Udelta}.
By assumption in Theorem~\ref{theo1}, $x_t \to x^\ast$ and $v_t\to 0$ as $t\to\infty$. Thus, for arbitrary $\delta>0$ we can choose $t_\delta >0$ such that $x_t\in \mathcal U_\delta$ and $\sup_{x \in \mathcal U_\delta}\|D^2 \Phi(x)\| \, \|v_t\| \le \frac{\delta}{8} $ for all $t \ge t_\delta$. We now consider the change of variables $(\tilde x_t)_{t \ge t_\delta} := (\Phi(x_t))_{t \ge t_\delta}$ and $(\tilde v_t)_{t \ge t_\delta} := (D\Phi(x_t)  v_t)_{t \ge t_\delta}$. 
Applying this change of variables to \eqref{eq:HBODE}, we obtain for all $t \ge t_\delta$
\begin{align} \label{eq:systemlin}
	\begin{split}
		\frac{d}{dt} \tilde x_t & = D\Phi(x_t) v_t = \tilde v_t ,\\
		\frac{d}{dt} \tilde v_t & = D^2 \Phi(x_t) [v_t, v_t] + D\Phi(x_t) \cdot (- \alpha v_t - \grad f(x_t)) \\
		&= -\alpha \tilde v_t -\grad \tilde f(\tilde x_t)  +\zeta_1(x_t,v_t)+ \zeta_2(x_t, \tilde x_t),
	\end{split}
\end{align}
where we have defined $\zeta_1(x_t,v_t):=D^2 \Phi(x_t) [v_t, v_t]$ with 
\begin{align*}
	\|\zeta_1(x_t,v_t)\| 
	\le \frac{\delta}{8} \|v_t\|\le \frac{\delta}{4} \|\tilde v_t\|,
\end{align*}
due to Lemma~\ref{lem:Udelta} (iv),
and $\zeta_2(x_t,\tilde x_t):=\nabla \tilde f(\tilde x_t)  -  D\Phi(x_t)   \grad f(x_t)$ satisfying
\begin{align} \begin{split} \label{eq:2374552}
		\| \zeta_2(x_t,\tilde x_t)\| &= \|\nabla \tilde f(\tilde x_t)  -  D\Phi(x_t)   \grad f(x_t)\|\\ & = \| (D\Phi^{-1}(\tilde x_t))^\top \,  \nabla f(x_t)  -  D\Phi(x_t) \grad f( x_t)\|\\ & \le \frac{\delta}{16 L}\|\grad  f(x_t)\| \le \frac{\delta}{8 L}\|\grad  \tilde f(\tilde x_t)\|\\
		&\le \frac{\delta}{8 L}\|\grad \tilde f(\proj_N(\tilde x_t)+\proj_T(\tilde x_t))-\grad\tilde f(\proj_T(\tilde x_t))\|\\
		&\le \frac{\delta}{4}\|\proj_N(\tilde x_t)\|,
	\end{split}
\end{align}
where in the second line we used Lemma~\ref{lem:Udelta} (i) and (iv), and in the last line we used $0 = \nabla \tilde f(\proj_T(\tilde x_t)) $ and  Lemma~\ref{lem:Udelta} (v). The ODE \eqref{eq:systemlin} is initialized by 
\begin{align*}
	(\tilde x_{t_\delta},\tilde v_{t_\delta})=(\Phi(x_{t_\delta}),D\Phi(x_{t_\delta})v_{t_\delta}),
\end{align*}
where $(x_t,v_t)_{t\ge 0}$ is a solution of \eqref{eq:HBODE}.
Next, we linearize $\nabla \tilde f$ around $0$ and observe that, by the choice of $\mathcal U_\delta$, see Lemma~\ref{lem:Udelta}, for all $\tilde x \in \Phi(\mathcal U_\delta)$ it holds that
\begin{align} \begin{split} \label{eq:9909}
		\|\nabla \tilde f(\tilde x) - \nabla^2 \tilde f(0) \, \proj_N(\tilde x)\| &\le \|\nabla \tilde f(\tilde x) - \nabla^2 \tilde f(\proj_T(\tilde x)) \, \proj_N(\tilde x)\| \\
		&+\| \nabla^2 \tilde f(\proj_T(\tilde x)) \, \proj_N(\tilde x) - \nabla^2 \tilde f(0) \, \proj_N(\tilde x)\| \\
		& \le \frac{\delta}{2} \|\proj_N(\tilde x)\|\,.
	\end{split}
\end{align}
In summary, the norms of the defined error terms $\zeta_1$ and $\zeta_2$, as well as the linearization error in \eqref{eq:9909}, depend only on $\|\tilde v_t\|$ and $\|\proj_N(\tilde x_t)\|$, but not on $\|\proj_T(\tilde x_t)\|$. This is crucial since the point $x^\ast \in \mathcal M$ is only contracting in the normal directions and there is no restoring force in the tangential directions. 

To proceed, we consider a restriction of the dynamical system \eqref{eq:systemlin}, where we omit the tangential directions of $(\tilde x_t)_{t \ge t_\delta}$. Using the computations above, we can write
\begin{align} 
	\begin{split}
		\label{eq:ODEproof}
		\frac{d}{dt} \proj_N(\tilde x_t) &= \proj_N(\tilde v_t) ,\\
		\frac{d}{dt} \tilde v_t 
		&= -\alpha \tilde v_t - \nabla^2 \tilde f(0) \proj_N(\tilde x_t) +\xi_t,
	\end{split}
\end{align}
for all $t \ge t_\delta$, where 
\[ 
	\xi_t:=\zeta_1(x_t,v_t) + \zeta_2(x_t,\tilde x_t) + \nabla^2 \tilde f(0) (\proj_N(\tilde x_t)) - \nabla \tilde f(\tilde x_t) 
\]
with
\begin{align} \label{eq:33333}
	\|\xi_t\| \le \delta\, \Bigl\| \begin{pmatrix}
		\proj_N(\tilde x_t) \\ \tilde v_t 
	\end{pmatrix} \Bigr\|\,.
\end{align}
Let $(x_t')_{t \ge t_\delta} \subset \R^{d_N}$ be defined via $\proj_N(\tilde x_t)=(0,x_t') \in \{0\}^{d_T} \times \R^{d_N}$ for all $t \ge t_\delta$.
We can rewrite the ODE \eqref{eq:ODEproof} in matrix notation as 
\begin{equation}\label{eq:ODEproof_matrix}
	\frac{d}{dt} \begin{pmatrix}
		x_t' \\ \tilde v_t
	\end{pmatrix} 
	=
	A
	\begin{pmatrix}
		x_t' \\ \tilde v_t
	\end{pmatrix} 
	+
	\begin{pmatrix}
		0 \\ \xi_t,
	\end{pmatrix}
\end{equation}
for initial $(x'_{t_\delta},\tilde v_{t_\delta})^\top\in\R^{d_N+d}$ with $\proj_N(\Phi(x_{t_\delta}))=(0,x'_{t_\delta})$ and $\tilde v_{t_\delta} = D\Phi(x_{t_\delta})v_{t_\delta}$, and a matrix $A \in \R^{(d_N+d)\times (d_N+d)}$ that can be written as
\begin{align*}
	A = 
	\begin{pmatrix}
		0 & 0 & \Id_{\R^{d_{N}}} \\
		0 & -\alpha \Id_{\R^{d_T}} & 0 \\ 
		-H & 0 & -\alpha \Id_{\R^{d_N}}
	\end{pmatrix},
\end{align*}
where $H \in \R^{d_N \times d_N}$ denotes the restriction of $\nabla^2 \tilde f(0)$ to $\{0\}^{d_T} \times \R^{d_N}$. Due to Lemma~\ref{lem:ftilde}, $H$ is a symmetric matrix satisfying $\spec(H) \subset [\mu,\infty)$, where $\spec(M)$ denotes the spectrum of a matrix $M$. 
Using the unique solution of \eqref{eq:ODEproof_matrix} given by 
\begin{align} \label{eq:11111}
	\begin{pmatrix}
		x_t' \\ \tilde v_t 
	\end{pmatrix} 
	= 
	\exp ((t-t_\delta) A)
	\begin{pmatrix}
		x_{t_\delta}' \\ \tilde v_{t_\delta}
	\end{pmatrix} 
	+ 
	\int_{t_\delta}^t \exp ((t-s) A) \begin{pmatrix}
		0 \\ \xi_s
	\end{pmatrix} ds\,,
\end{align}
we derive the following properties.

\begin{lemma}
	Let $\alpha,\delta>0$ be arbitrary and $(x'_t,\tilde v_t)_{t\ge t_\delta}$ be the unique solution of \eqref{eq:ODEproof_matrix} with initial $(x'_{t_\delta},\tilde v_{t_\delta})\in\R^{d_N} \times \R^d$. Then
	\begin{enumerate}
		\item[(i)] $ -\max_{\rho \in \spec(A)}\operatorname{Re}(\rho) \ge  \frac 12 (\alpha-\sqrt{\max(0,\alpha^2-4\mu)}):= m(\alpha)>0$, and,
		\item[(ii)] for any $\eta\in(-m(\alpha),0)$ there exists a constant $C_\eta\ge0$ such that
		\begin{equation} \label{eq:key}
			\Bigl\| \begin{pmatrix}
				x_t' \\ \tilde v_t
			\end{pmatrix} \Bigr\| \le C_\eta \Bigl\| \begin{pmatrix}
				x_{t_\delta}' \\ \tilde v_{t_\delta}
			\end{pmatrix} \Bigr\| e^{(\eta + C_\eta  \delta)(t-t_{\delta})} ,\quad \text{ for all } t \ge t_\delta\,.
		\end{equation}
	\end{enumerate}
\end{lemma}
\begin{proof}
	We begin with the first claim (i) 
	and compute the eigenvalues of $A$. First note that for $i=d_N+1, \dots, d$ the $i$-th vector in the standard basis of $\R^{d+d_N}$, denoted as $e_i$, is an eigenvector of $A$ with $A e_i = - \alpha e_i$. Thus, $-\alpha$ is an eigenvalue of $A$ and we can restrict our attention to computing the spectrum of 
	\begin{align*}
		\tilde A :=     \begin{pmatrix}
			0 & \Id_{\R^{d_{N}}} \\
			-H & -\alpha \Id_{\R^{d_N}}
		\end{pmatrix}.
	\end{align*}
	Since $H$ is a symmetric and positive definite matrix with eigenvalues bounded from below by $\mu$, see Lemma~\ref{lem:ftilde}, there exists an orthogonal matrix $Q \in \R^{d_N \times d_N}$ with $Q^\top Q= \Id_{R^{{d_N}}}$ and
	\begin{align*}
		H = Q^\top \diag(\lambda_1, \dots, \lambda_{d_N}) Q\,,
	\end{align*}
	where $ \diag(\lambda_1, \dots, \lambda_{d_N})$ is the diagonal matrix with entries $\lambda_1, \dots, \lambda_{d_N} \in [\mu, \infty)$. 
		Consider the matrix $\tilde Q \in \R^{2d_N \times 2 d_N}$ given by
		\begin{align*}
			\tilde Q = \begin{pmatrix}
				Q & 0 \\
				0 & Q.
			\end{pmatrix}
		\end{align*}
		Then, $\tilde Q^\top \tilde Q = \Id_{\R^{2d_N}}$ and 
		\begin{align*}
			\spec(\tilde A) = \spec(\tilde Q^\top \tilde A \tilde Q) = \spec \Bigl( \begin{pmatrix}
				0 & \Id_{\R^{d_{N}}} \\
				-\diag(\lambda_1, \dots, \lambda_{d_N}) & -\alpha \Id_{\R^{d_N}}
			\end{pmatrix} \Bigr).
		\end{align*}
	Hence, it clearly suffices to compute the eigenvalues of 
	\begin{align*}\tilde A_i:=
		\begin{pmatrix}
			0 & 1 \\
			-\lambda_i & - \alpha 
		\end{pmatrix}
	\end{align*}
	for all $i = 1, \dots, d_N$. A straightforward computation shows that the eigenvalues of the latter matrix are given by
	\begin{align*}
		\eta_i^{(1)}=-\frac{1}{2} (\alpha-\sqrt{\alpha^2-4\lambda_i}) \quad \text{ and }\quad \eta_i^{(2)}= -\frac{1}{2} (\alpha+\sqrt{\alpha^2-4\lambda_i}). 
	\end{align*}
	Altogether, we find that $\spec(A)=\{- \alpha, \eta_i^{(1)},\eta_i^{(2)}:i=1, \dots, d_N\}$ so that \[\max_{\rho \in \spec(A)} \operatorname{Re} (\rho) \,  \le  -\frac 12(\alpha-\sqrt{\max(0,\alpha^2-4\mu)})=:- m(\alpha)\,.\]
	
	Next, we consider the second claim (ii). 
	Let $\eta\in(-m(\alpha),0)$. Then (see, e.g.~\cite[Lemma~7]{polyak1964some}),
	there exists a $C_\eta \ge 0$ such that 
	\begin{align} \label{eq:22222}
		\sup_{t \ge 0} e^{-\eta t} \|e^{A t} \| \le C_\eta.
	\end{align}
	Define $(g(t))_{t \ge t_\delta}$ via $g(t) := e^{-\eta (t-t_\delta)} \Bigl\| \begin{pmatrix}
		x_t' \\ \tilde v_t
	\end{pmatrix} \Bigr\|$. Combining \eqref{eq:33333}, \eqref{eq:11111} and~\eqref{eq:22222} one obtains for all $t \ge t_\delta$ the inequality 
	\begin{align*}
		g(t) &\le e^{-\eta (t-t_\delta)} \|\exp ((t-t_\delta) A)\|  \Bigl\| \begin{pmatrix}
			x_{t_\delta}' \\ \tilde v_{t_\delta}
		\end{pmatrix} \Bigr\| 
		+
		e^{-\eta (t-t_\delta)}  \int_{t_\delta}^t \|\exp ((t-s) A)\| \, \Bigl\|\begin{pmatrix}
			0 \\ \zeta_s,
		\end{pmatrix}\Bigr\| \, ds \\
		& \le C_\eta \Bigl\| \begin{pmatrix}
			x_{t_\delta}' \\ \tilde v_{t_\delta}
		\end{pmatrix} \Bigr\|
		+  C_\eta \int_{t_\delta}^t e^{- \eta (s-t_\delta) } \Bigl\|\begin{pmatrix}
			0 \\ \zeta_s,
		\end{pmatrix}\Bigr\| \, ds \\
		& \le  C_\eta \Bigl\| \begin{pmatrix}
			x_{t_\delta}' \\ \tilde v_{t_\delta}
		\end{pmatrix} \Bigr\|
		+  C_\eta \delta \int_{t_\delta}^t  g(s) \, ds.
	\end{align*}
	Using Gronwall's inequality, we get for all $t\ge t_\delta$ that
	\begin{align*}
		g(t) \le C_\eta \Bigl\| \begin{pmatrix}
			x_{t_\delta}' \\ \tilde v_{t_\delta}
		\end{pmatrix} \Bigr\| e^{C_\eta \delta (t-t_\delta)},
	\end{align*}
	such that
	\begin{align*}
		\Bigl\| \begin{pmatrix}
			x_t' \\ \tilde v_t
		\end{pmatrix} \Bigr\| \le C_\eta \Bigl\| \begin{pmatrix}
			x_{t_\delta}' \\ \tilde v_{t_\delta}
		\end{pmatrix} \Bigr\| e^{(\eta + C_\eta  \delta)(t-t_{\delta})}\,.
	\end{align*}
\end{proof}

\subsection{Final step} \label{sec:final} 
Now, we are ready to finish the proof of Theorem~\ref{theo1}. 
\begin{proof}[Proof of Theorem~\ref{theo1}]
	Note that
	\begin{align*}
		\|x_t-x^\ast\| &=\lim_{r\to\infty} \|x_t-x_r\| \le \lim_{r\to\infty} \int_t^r \|v_s\|\,ds \le \int_t^\infty \|v_s\| \, ds \le 2\int_t^\infty \|\tilde v_s\| \, ds \, ,
	\end{align*}
	where we have used that $(x_t)_{t\ge0}$ converges as $t\to\infty$, $x \mapsto \|x_t-x\|$ is continuous and Lemma~\ref{lem:Udelta} (iv).
	Let $\eps \in (0,m(\alpha))$ be arbitrary. For the choice $\eta := -m(\alpha) + \frac{\eps}{3}$ and $\delta = \frac{\eps}{3 C_\eta}$ in \eqref{eq:key} we obtain
	\begin{align} \begin{split} \label{eq:236452672}
			\limsup\limits_{t \to \infty} e^{( m(\alpha) - \eps) t} &\|x_t-x^\ast\| \le  \limsup\limits_{t \to \infty} 2 e^{( m(\alpha) - \eps) t}  \int_t^\infty  \Bigl\| \begin{pmatrix}
				x_s' \\ \tilde v_s
			\end{pmatrix} \Bigr\| \, ds \\
			& \le 2 C_{\eta} \Bigl\| \,  \begin{pmatrix}
				x_{t_\delta}' \\ \tilde v_{t_\delta}
			\end{pmatrix} \Bigr\|e^{(m(\alpha)-\eps)t_\delta} \Bigl( \limsup\limits_{t \to \infty} \int_t^\infty e^{-\frac{\eps}{3}(s-t_{\delta})}\, ds \Bigr)=0.
		\end{split}
	\end{align}
	This proves the first assertion.
	
	Similarly, we are able to compute the convergence rate for $(f(x_t))_{t \ge 0}$. Recall that, due to Lemma~\ref{lem:Udelta} (iii), one obtains for $\delta \le 4L$ and for all $x \in \Phi(\mathcal U_\delta) \cap (\R^{d_T} \times \{0\}^{d_N})$
	\begin{align*}
		\langle \nabla^2 \tilde f(x) v,v \rangle \le \langle \nabla^2 \tilde f(0) v,v \rangle +\frac{\delta}{4} \|v\|^2  \le 2 L \|v\|^2 \, .
	\end{align*}
	Using $\nabla \tilde f(x)=0$ for $x \in \Phi(\mathcal U_\delta) \cap (\R^{d_T} \times \{0\}^{d_N})$ and $\tilde f \in C^2$, we can apply Taylor's theorem and use \eqref{eq:key} for the choice $\eta := -m(\alpha) + \frac{\eps}{3}$ and $\delta = \min(4L, \frac{\eps}{3 C_\eta})$ to deduce the second assertion, 
	\begin{align} \label{eq:993725}
		\begin{split}
			\limsup_{t \to \infty} e^{2( m(\alpha) - \eps) t}  (f(x_t)- f(x^\ast)) & = \limsup_{t \to \infty} e^{2( m(\alpha) - \eps) t}  ( \tilde f(\tilde x_t)- \tilde f(0))  \\
			& \le \limsup_{t \to \infty} e^{2( m(\alpha) - \eps) t}  4 L \,  \|\proj_N (\tilde x_t)\|^2 =0\,.
		\end{split}
	\end{align}
\end{proof}

\section{Proof of Theorem~\ref{theo2}} \label{sec:discrete}
In this section, we now turn to the proof of Theorem~\ref{theo2}, which establishes the accelerated local convergence of Polyak's heavy ball method \eqref{eq:HBdiscreteintro} in discrete time. Let $x^\ast$ be a $(\mu,L)$-regular point and, throughout this section, assume without loss of generality that $f(x^\ast)=0$. Moreover, without loss of generality, we set $N=0$ in the statement of Theorem~\ref{theo2}.

\paragraph{Structure of the proof} Contrary to the continuous time setting, there exists no result that directly proves convergence of the sequence $(x_n)_{n \in \N_0}$ for all pairs of hyperparameters $\beta \in (0,1)$ and $\gamma\in (0,\frac{2(1+\beta)}{L})$.
For example, the axiomatic Lyapunov-based convergence framework presented in \cite{attouch2013convergence, ochs2018local} guarantees local convergence of the iterates $(x_n)_{n \in \N_0}$ only in the overdamped regime, namely for the hyperparameters $\beta \in (0,1)$ and $\gamma \in (0,\frac{2(1-\beta)}{L})$.
Moreover, there exists a non-empty subset of the hyperparameters that are covered by Theorem~\ref{theo2} such that convergence cannot be guaranteed for arbitrary initial condition $x_0,x_1 \in \R^d$, see~\cite{goujaud2023provable}. Therefore, our analysis focuses on local convergence, and the structure of the proof is as follows. 
\begin{itemize} 
	\item \textit{Local convergence:} First, in Lemma~\ref{lem:conv}, we show that there exists a neighborhood $\mathcal U_\delta$ of $x^\ast$ such that if $(x_n)_{n \in \N_0}$ remains in $\mathcal U_\delta$ for all $n \in \N_0$, then $(x_n)_{n \in \N_0}$ converges to a local minimum in $\mathcal U_\delta$, which we denote by $x_\infty$. We also derive a suboptimal a-priori rate of convergence.
	\item \textit{Local attraction:} Next, in Lemma~\ref{lem:conv2}, we prove that there exists an $r>0$ with $B_r(x^\ast) \subset \mathcal U_\delta$ such that $x_0,x_1 \in  B_r(x^\ast)=:V$ implies that $(x_n)_{n\in\N_0}$ stays in $\mathcal U_\delta$ and thus, using the first step, converges to $x_\infty$. 
	\item \textit{Local acceleration:} Finally, we carry out a refined local analysis in the neighborhood of $x_\infty$ and prove the asymptotically optimal rates for Polyak's heavy ball method. This finishes the proof of Theorem~\ref{theo2}.
\end{itemize}
For the remaining section, we fix $\delta >0$ and denote by $\mathcal U_\delta$ a neighborhood of $x^\ast$ that satisfies the conclusion of Lemma~\ref{lem:Udelta}. 
We will show that there exists a $\delta >0$ such that $x_n\in \mathcal U_\delta$ for all $n\in\N_0$ implies that $(x_n)_{n \in \N_0}$ converges to a local minimum $x_\infty \in \mathcal U_\delta$.
For sufficiently small $\delta>0$, we can assume without loss of generality that $\mathcal U_\delta \subset \overline{B_\delta(x^\ast)} \subset \mathcal U$, where $\mathcal U$ denotes the set given by Lemma~\ref{lem:chart}, and there exists a constant $C_\Phi \ge 0$ such that $D\Phi$ is $C_\Phi$-Lipschitz continuous on $\overline{B_\delta(x^\ast)}$. The last assumption can be satisfied since $\Phi: \mathcal U \to \Phi(\mathcal U)$ is $C^2$. Furthermore, after further shrinking $\mathcal U_\delta$, we can assume that $\mathrm{conv}(\mathcal U_\delta) \subset \overline{B_\delta(x^\ast)} $ and that Lemma~\ref{lem:Udelta} (iv) and $\|\nabla^2 f(x)\|\le L$ also holds for convex combinations of points in $\mathcal U_\delta$, respectively convex combinations of points in $\Phi(\mathcal U_\delta)$.

\subsection{Technical preparation} Consider the dynamics under the chart $\Phi$ defined as $(\tilde x_n)_{n \in \N_0}:= (\Phi(x_n))_{n \in \N_0}$. Since 
\begin{align*}
	D\Phi(x_n)(x_{n+1}-x_{n}) = -\gamma D\Phi(x_n)\, \nabla f(x_{n}) + \beta D\Phi(x_n) (x_n-x_{n-1}),
\end{align*}
we can write 
\begin{align*}
	\tilde x_{n+1} = \Phi(x_{n+1}) &= \Phi(x_n) -\gamma \nabla \tilde f(x_n) + \beta (\Phi(x_n)-\Phi(x_{n-1})) + \zeta(x_n,x_{n-1}, \tilde x_n),\\
	&= \tilde x_n - \gamma \nabla \tilde f(x_n) + \beta (\tilde x_n-\tilde x_{n-1}) + \zeta(x_n,x_{n-1}, \tilde x_n),
\end{align*}
with $\zeta(x_n,x_{n-1}, \tilde x_n) :=  \zeta_1(x_n,x_{n-1})+ \zeta_2(x_n, \tilde x_n)+\zeta_3(x_n,x_{n-1})$ for 
\begin{align*}
	\zeta_1(x_n,x_{n-1}) &:= \Phi(x_{n+1})-\Phi(x_{n})-D\Phi(x_{n})(x_{n+1}-x_{n}), \\
	\zeta_2(x_n, \tilde x_n) &:= \gamma ( \nabla \tilde f(\tilde x_n) - D\Phi(x_n)\, \nabla f(x_{n})) \quad \text{ and } \\
	\zeta_3(x_n,x_{n-1}) &:= \beta (D\Phi(x_n) (x_n-x_{n-1})-\Phi(x_n)+\Phi(x_{n-1})).
\end{align*}

Due to the Lipschitz continuity of $D\Phi$ and $\nabla f$, a Taylor expansion shows that for all $n \in \N$
\begin{align*}
	\|\zeta_1(x_n,x_{n-1})\| &= \|\Phi(x_{n+1})-\Phi(x_{n})-D\Phi(x_{n})(x_{n+1}-x_{n})\|\\ &\le \frac{C_\Phi}{2} \|x_{n+1}-x_{n}\|^2 
	\\
	&\le C_\Phi \gamma^2 \|\nabla f(x_n)\|^2 + C_\Phi \beta^2 \|x_{n}-x_{n-1}\|^2 \\
	& \le C_\Phi \gamma^2 L^2 d(x_n,\mathcal M)^2 + C_\Phi \beta^2 \|x_{n}-x_{n-1}\|^2 
\end{align*}
Note that in the last inequality we have used that \[\|\nabla f(x_n)\|^2=\|\nabla f(x_n)-\nabla f(x^\ast)\|^2\le L^2\|x_n-x^\ast\|^2,\] since $\|\nabla^2 f(x)\|\le L$ for all points $x$ in the segment between $x_n$ and $x^\ast$. Using Lemma~\ref{lem:Udelta} (iv) one gets 
\begin{align} \begin{split} \label{eq:28663}
		&\|x_n-x_{n-1}\|^2 = \|\Phi^{-1}(\Phi(x_n))-\Phi^{-1}(\Phi(x_{n-1}))\|^2 \\
		& \le \Bigl( \int_0^1 \|D\Phi^{-1}(t\Phi(x_n)+(1-t) \Phi(x_{n-1}))\| \, \|\Phi(x_n)-\Phi(x_{n-1})\| \, dt \Bigr)^2 \\
		&\le 4\|\tilde x_n- \tilde x_{n-1}\|^2. 
	\end{split}
\end{align}

Analogously to \eqref{eq:2374552} one has $\| \zeta_2(x_n,\tilde x_n)\| \le \gamma \frac{\delta}{4}  \|\proj_N(\tilde x_n)\|$.
Moreover, by the $C_\Phi$-Lipschitz continuity of $D\Phi$ on $\overline{B_\delta(x^\ast)}$ we can bound
\begin{align*}
	\|\zeta_3(x_n,x_{n-1})\| &\le \beta \int_0^1 \|D\Phi(t x_n + (1-t)x_{n-1})-D\Phi(x_n)\| \, \|x_n-x_{n-1}\| \, dt \\
	& \le \frac{\beta}{2} C_\Phi \|x_n-x_{n-1}\|^2 \le 2 \beta C_\Phi \|\tilde x_n-\tilde x_{n-1}\|^2,
\end{align*}
where in the last inequality we used \eqref{eq:28663}.
Together we obtain
\begin{align*}
	\|\zeta(x_n,x_{n-1}, \tilde x_n)\| \le & C_\Phi \gamma^2 L^2 \|\proj_N(\tilde x_n)\|^2 +\gamma \frac{\delta}{4}  \|\proj_N(\tilde x_n)\| \\
	&+ 2\beta C_\Phi(2\beta+1) \|\tilde x_n-\tilde x_{n-1}\|^2.
\end{align*}
As in \eqref{eq:9909} we can bound the norm of $\zeta_4(\tilde x_n):= \nabla \tilde f(\tilde x_n) - \nabla^2 \tilde f(0) (\proj_N(\tilde x_n))$ by 
$
\| \zeta_4(\tilde x_n) \|\le \frac{\delta}{2}\| \proj_N (\tilde x_n)\|
$
so that we can write
\begin{equation} \label{eq:2h44}
	\tilde x_{n+1} = \tilde x_n -\gamma \nabla^2 \tilde f(0)(\proj_N (\tilde x_n)) + \beta (\tilde x_n-\tilde x_{n-1}) + \xi_n,
\end{equation}
for $\xi_n := \zeta(x_n,x_{n-1}, \tilde x_n) + \zeta_4(\tilde x_n)$
with 
\begin{align*}\begin{split} 
		\|\xi_n\| \le &C_\Phi \gamma^2 L^2 \|\proj_N(\tilde x_n)\|^2 +(\gamma+2) \frac{\delta}{4}  \|\proj_N(\tilde x_n)\| \\
		&+ 2C_\Phi\beta (2\beta+1) \|\tilde x_n-\tilde x_{n-1}\|^2.
	\end{split}
\end{align*}
Recall that $\mathcal U_\delta \subset \overline{B_\delta(x^\ast)}$ so that using \eqref{eq:28663} with reversed roles one gets
\begin{align*}
	\|\tilde x_n-\tilde x_{n-1}\| \le 2\|x_n-x_{n-1}\| \le 4\delta \quad \text{ and }\quad  \|\proj_N(\tilde x_n)\| \le d(x_n,\mathcal M ) \le \delta\, .
\end{align*}
Therefore, there exists a constant $C_\delta>0$ that satisfies $C_\delta \to 0$ as $\delta \to 0$ such that 
\begin{align}\begin{split} \label{eq:008}
		\|\xi_n\|
		& \le C_\delta \max(\|\proj_N(\tilde x_n)\|,\|\proj_N(\tilde x_{n-1})\|,\|\proj_T(\tilde x_n-\tilde x_{n-1})\|)\, .
	\end{split}
\end{align}

Similarly to the continuous time analysis, we separate the tangential directions from the normal directions. In the normal directions, the dynamical system behaves similar to the strongly convex case considered in \cite{polyak1964some}. In the tangential directions there is no restoring force, so that we cannot derive a convergence rate with the classical approach. However, since the friction parameter $\beta$ satisfies $0 < \beta < 1$ we will see that the momentum vector $(\proj_T(x_n)-\proj_T(x_{n-1}))_{n \in \N_0}$ vanishes exponentially fast. 

In the following, we decompose the vector $x_n=:(x_n'',x_n')\in\R^d$ into its component in the normal direction $x_n'\in\R^{d_N}$ and its component in the tangential direction $x_n''\in\R^{d_T}$.
Using \eqref{eq:2h44} together with Lemma~\ref{lem:ftilde}, we can write 
\begin{align*}
	\begin{pmatrix}
		x_{n+1}' \\ x_n' \\ x_{n+1}''-x_n''
	\end{pmatrix}
	= \begin{pmatrix} x'_{n}-\gamma H ( x'_n))+\beta\left( x'_{n}-x'_{n-1}\right) \\  x'_{n} \\\ \beta (x''_{n}-x''_{n-1}) 
	\end{pmatrix}
	+
	\begin{pmatrix}
		\xi_n' \\
		0 \\  \xi_n''
	\end{pmatrix},
\end{align*}
for all $n \in \N_0$, where $H$ denotes the restriction of $\nabla^2 \tilde f(0)$ to $\{0\}^{d_T} \times \R^{d_N}$, which is a symmetric matrix with $\spec(H) \subset [\mu, L]$.
In other words,
\begin{align} \label{eq:324805723648582}
	\begin{pmatrix}
		x_{n+1}' \\ x_n' \\ x_{n+1}''-x_n''
	\end{pmatrix}
	= A
	\begin{pmatrix}
		x_{n}' \\ x_{n-1}' \\ x_{n}''-x_{n-1}''
	\end{pmatrix}
	+
	\begin{pmatrix}
		\xi_n' \\
		0 \\  \xi_n''
	\end{pmatrix},
\end{align}
where
\begin{align} \label{eq:A}
	A := \begin{pmatrix} 
		(1+\beta) \Id_{R^{d_N}}-\gamma H & -\beta \Id_{\R^{d_N}} & 0 \\
		\Id_{R^{d_N}} & 0 & 0 \\
		0 & 0 & \beta
	\end{pmatrix}.
\end{align}
Next, we compute the eigenvalues of $A$ defined in \eqref{eq:A}.
\begin{lemma} \label{lem:ev}
	Let $\beta \in (0,1)$ and $\gamma \in \bigl(\frac{2(1+\beta)}{L}\bigr)$. Then, $\rho(A) \le m(\gamma, \beta)$, where 
	\begin{align*}
		m&(\gamma, \beta) := \\
		&\begin{cases}
			\sqrt \beta , & \text{ if } \gamma \in \mathcal I_{\beta,\mu, L}, \\
			\frac{1+\beta-\gamma \mu}{2}+\sqrt{\left(\frac{1+\beta-\gamma \mu}{2}\right)^{2}-\beta}, & \text{ if }  \gamma \in \Bigl(0, \frac{2(1+\beta)}{L+\mu} \Bigr] \setminus \mathcal I_{\beta,\mu, L}, \\
			\frac{\gamma L-(1+\beta)}{2}+\sqrt{\left(\frac{\gamma L-(1+\beta)}{2}\right)^{2}-\beta}, & \text{ if } \gamma \in \Bigl( \frac{2(1+\beta)}{L+\mu} , \frac{2(1+\beta)}{L} \Bigr) \setminus \mathcal I_{\beta,\mu, L},\\
		\end{cases}
	\end{align*}  
	with $\mathcal I_{\beta,\mu, L}:= \Big[\frac{(1-\sqrt{\beta})^2}{\mu}, \frac{(1+\sqrt{\beta})^2}{L}\Big]$.
	Moreover, one has $m(\gamma, \beta)<1$ and $m(\gamma, \beta)$ is minimized by choosing
	\begin{align*}
		\gamma=\frac{4}{(\sqrt{\mu}+\sqrt{L})^{2}} \quad \text { and } \quad \beta= \Bigl( \frac{\sqrt \kappa-1}{\sqrt{\kappa}+1}\Bigr)^2
	\end{align*}
	for which $m(\gamma, \beta) = \frac{\sqrt{\kappa}-1}{\sqrt{\kappa}+1}$.
\end{lemma}

\begin{proof}
	First note that for $i=2 d_N+1, \dots, d_N+d$ the $i$-th vector in the standard basis of $\R^{d_N+d}$, 
	denoted by $e_i$, is an eigenvector of $A$ with $A e_i = \beta e_i$. Thus, $\beta$ is an eigenvalue of $A$ and we can restrict our attention to computing the spectrum of 
	\begin{align*}
		\tilde A :=     \begin{pmatrix} 
			(1+\beta) \Id_{R^{d_N}}-\gamma H & -\beta \Id_{\R^{d_N}}  \\
			\Id_{R^{d_N}} & 0  
		\end{pmatrix}.
	\end{align*}
	Observe that $T$ is self-similar to a block diagonal matrix
	\begin{align*}
		\widehat{A}=\left(\begin{array}{ccc}
			A_{1} & & \\
			& \ddots & \\
			& & A_{d}
		\end{array}\right)
	\end{align*}
	with blocks
	\begin{align*}
		A_{i}=\left(\begin{array}{cc}
			1+\beta-\gamma \lambda_{i} & -\beta \\
			1 & 0
		\end{array}\right) \in \mathbb{R}^{2 \times 2},
	\end{align*}
	where $\lambda_1, \dots, \lambda_N$ denote the eigenvalues of $H$. Therefore, it suffices to compute the eigenvalues of $A_i$ for all $i =1, \dots, d_N$. A straightforward computation shows that the eigenvalues of $A_i$ are given by
	\begin{align*}
		& \eta_{i}^{(1)}=\frac{1+\beta-\gamma \lambda_{i}}{2}-\sqrt{\left(\frac{1+\beta-\gamma \lambda_{i}}{2}\right)^{2}-\beta} \\
		& \eta_{i}^{(2)}=\frac{1+\beta-\gamma \lambda_{i}}{2}+\sqrt{\left(\frac{1+\beta-\gamma \lambda_{i}}{2}\right)^{2}-\beta}.
	\end{align*}
	Using $\lambda_1, \dots, \lambda_N \in [\mu, L]$ we get for $\gamma \le \frac{2(1+\beta)}{L+\mu}$ that
	\begin{align*}
		\rho(A) \le \rho\Bigl(\left(\begin{array}{cc}
			1+\beta-\gamma \mu & -\beta \\
			1 & 0
		\end{array}\right)\Bigr)
	\end{align*}
	and for $\gamma \ge \frac{2(1+\beta)}{L+\mu}$ that
	\begin{align*}
		\rho(A) \le \rho\Bigl(\left(\begin{array}{cc}
			1+\beta-\gamma L & -\beta \\
			1 & 0
		\end{array}\right)\Bigr).
	\end{align*}
	Now, for $i \in \{1, \dots, d_N\}$ with $\left(\frac{1+\beta-\gamma \lambda_{i}}{2}\right)^{2}-\beta \leq 0$ one has $|\eta_{i}^{(1)}|=|\eta_{i}^{(2)}|=\sqrt{\beta}$. Thus, for the choice 
	\begin{align*}
		\frac{(1-\sqrt{\beta})^2}{\mu} \le \gamma \le \frac{(1+\sqrt{\beta})^2}{L}
	\end{align*}
	one has $\rho(A)= \sqrt{\beta}$. 
	In conclusion,
	$
	\rho(A) \le m(\gamma, \beta)
	$.
	The last statement is a straight-forward computation. 
\end{proof}

\subsection{Local convergence} 
In this section, we prove local convergence of the heavy ball method for all pairs of hyperparameters $\beta \in (0,1)$ and $\gamma\in (0,\frac{2(1+\beta)}{L})$ conditioned on staying in the set $\mathcal U_\delta$. For this, we use the representation \eqref{eq:324805723648582} and derive an a-priori rate of convergence in terms of the spectral radius of $A$ defined in \eqref{eq:A}.
\begin{lemma} \label{lem:aux}
	For every $\eps >0 $ there exist  $\delta, C >0 $ such that for all $n \in \N_0$ with $(x_i)_{i=0, \dots, n}\subset \mathcal U_\delta$ one has
		\begin{align*}
			(\rho(A)+\frac{7}{9}\eps)^{-n} \, \Bigl\| \begin{pmatrix}
				x_{n+1}' \\ x_n' \\ x_{n+1}''-x_n''
			\end{pmatrix} \Bigr\|  \le C,
		\end{align*}
	where $\rho(A):=\sup\{|\lambda| : \lambda \in  \spec(A)\}$  denotes the spectral radius of the matrix $A$.     
\end{lemma}

\begin{proof}
	Let $n \in \N_0$ with $(x_i)_{i=0, \dots, n}\subset \mathcal U_\delta$. Then, by the variation of constant formula, one has
	\begin{align} \label{eq:28664}
		\begin{pmatrix}
			x_{n+1}' \\ x_n' \\ x_{n+1}''-x_n''
		\end{pmatrix}
		= A^n
		\begin{pmatrix}
			x_{1}' \\ x_{0}' \\ x_{1}''-x_{0}''
		\end{pmatrix}
		+\sum_{i=1}^{n}
		A^{n-i}
		\begin{pmatrix}
			\xi_i' \\
			0 \\  \xi_i''
		\end{pmatrix}.
	\end{align}
	Let $\eta > \rho(A)$. By Gelfand's formula \cite{gelfand} (see also \cite{foucart2018matrix}), there exists a constant $C_\eta>0$ such that
	\begin{align}  \label{eq:009}
		\sup_{n \in \N_0} \eta^{-n} \|A^n\| \le C_\eta.
	\end{align}
	Define $(g(n))_{n \in \N_0}$ via $g(n):= \eta^{-n} \Bigl\| \begin{pmatrix}
		x_{n+1}' \\ x_n' \\ x_{n+1}''-x_n''
	\end{pmatrix} \Bigr\|$. Combining \eqref{eq:28664}, \eqref{eq:009} and \eqref{eq:008}, we achieve
	\begin{align*}
		g(n) &\le C_\eta \Bigl\| \begin{pmatrix}
			x_{1}' \\ x_0' \\ x_{1}''-x_0''
		\end{pmatrix} \Bigr\| + C_\eta \sum_{i=1}^{n-1} \eta^{-i} \Bigl\| \begin{pmatrix}
			\xi_i' \\
			0 \\  \xi_i''
		\end{pmatrix} \Bigr\| \\
		& \le  C_\eta g(0) + C_\eta C_\delta \sum_{i=1}^{n-1} g(i),
	\end{align*}
	so that $g(n)\le C_\eta g(0) (1+C_\eta C_\delta )^n$ by Gronwall's inequality in discrete time. Thus, we have verified that
	\begin{align} \label{eq:22221445}
		\Bigl\| \begin{pmatrix}
			x_{n+1}' \\ x_n' \\ x_{n+1}''-x_n''
		\end{pmatrix} \Bigr\| \le C_\eta  \Bigl\| \begin{pmatrix}
			x_{1}' \\ x_0' \\ x_{1}''-x_0''
		\end{pmatrix} \Bigr\| (1+C_\eta C_\delta)^n\eta^n.
	\end{align}
	Now, let $0< \eps < 1$ be arbitrary. Then, \eqref{eq:22221445} for the choice $\eta= \rho(A)+\frac{\eps}{3}$ and $C_\delta = \frac{\eps}{3 C_\eta \max(1,\rho(A)) }$ gives 
	\begin{align*}
		\Bigl\| \begin{pmatrix}
			x_{n+1}' \\ x_n' \\ x_{n+1}''-x_n''
		\end{pmatrix} \Bigr\| &\le C_\eta  \Bigl\| \begin{pmatrix}
			x_{1}' \\ x_0' \\ x_{1}''-x_0''
		\end{pmatrix} \Bigr\| \,  \Bigl(\Bigl(1+\frac{\eps}{3 \max(1,\rho(A))}\Bigr)  \Bigl(\rho(A)+\frac{\eps}{3}\Bigr)\Bigr)^{n} \\
		&  \le C_\eta  \Bigl\| \begin{pmatrix}
			x_{1}' \\ x_0' \\ x_{1}''-x_0''
		\end{pmatrix} \Bigr\| \, \Bigl(\rho(A)+\frac{7}{9}\eps\Bigr)^{n}\,.
	\end{align*}
	This implies that
	\begin{align*}
		(\rho(A)+\frac{7}{9}\eps)^{-n} \Bigl\| \begin{pmatrix}
			x_{n+1}' \\ x_n' \\ x_{n+1}''-x_n''
		\end{pmatrix} \Bigr\|  \le C_\eta  \Bigl\| \begin{pmatrix}
			x_{1}' \\ x_0' \\ x_{1}''-x_0''
		\end{pmatrix} \Bigr\|.
	\end{align*}
\end{proof}

We are now able to prove convergence of Polyak's heavy ball method assuming that $(x_n)_{n \in \N_0} \subset \mathcal U_\delta$ for a sufficiently small $\delta>0$.
\begin{lemma} \label{lem:conv}
	Assume that $\rho(A)<1$. Then there exists a $\delta>0$ such that if $(x_n)_{n \in \N_0} \subset \mathcal U_\delta$, then the sequence $(x_n)_{n\in\N_0}$
	converges to a local minimum $x_\infty \in \mathcal U$, where $\mathcal U$ denotes the set that is given by Lemma~\ref{lem:chart}.
	Moreover, for every $\eps>0$ there exists a $\delta>0$ such that if $(x_n)_{n \in \N_0} \subset \mathcal U_\delta$ one has
	\begin{align*}
		\limsup_{n \to \infty}  (\rho(A)+\eps)^{-n} \|x_n-x_\infty\| =0\,.
	\end{align*}
\end{lemma}

\begin{proof}
	Note that, by definition, $\tilde x_n = (x_n'',x_n')$ for all $n \in \N_0$. Therefore, we can bound the increments of $(\tilde x_n)_{n\in\N_0}$ by
	\begin{align} \label{eq:23372987r9297}
		\| \tilde x_{n+1}-\tilde x_{n} \|^2 \le 2 \,  \Bigl\| \begin{pmatrix}
			x_{n+1}' \\ x_n' \\ x_{n+1}''-x_n''
		\end{pmatrix} \Bigr\|^2.
	\end{align}
	Let $0< \eps < 1-\rho(A)$. Using Lemma~\ref{lem:aux}, there exists a $\delta>0$ such that $(x_n)_{n \in \N_0} \subset \mathcal U_\delta$ implies 
	\begin{align*}
		\sum_{i=n}^\infty \| \tilde x_{i+1}-\tilde x_{i} \|&\le \sqrt 2 \,  \sum_{i=n}^\infty \Bigl\| \begin{pmatrix}
			x_{i+1}' \\ x_i' \\ x_{i+1}''-x_i''
		\end{pmatrix} \Bigr\|
		\\
		&\le C_\eta  \Bigl\| \begin{pmatrix}
			x_{1}' \\ x_0' \\ x_{1}''-x_0''
		\end{pmatrix} \Bigr\| \, \sum_{i=n}^\infty \Bigl( \rho(A)+\frac{7}{9}\eps\Bigr)^{i} \to 0\,,
	\end{align*}
	as $n\to\infty$. Hence, $(x_n)_{n \in \N_0}$ converges to a point $x_\infty \in \overline{\mathcal U_\delta} \subset \mathcal U$. Since $x_n' \to 0$ and $\Phi^{-1}(\Phi(\mathcal U_\delta) \cap (\R^{d_T} \times \{0\}^{d_N} ))\subset \mathcal M$, $x_\infty$ is a local minimum. Using $\|x_n-x_\infty\| \le 2\|\tilde x_n-\Phi(x_\infty)\|$, see \eqref{eq:28663}, we get
	\begin{align*}
		\limsup_{n \to \infty}  (\rho(A)+\eps)^{-n} \|x_n-x_\infty\|  &\le \limsup_{n \to \infty} 2 (\rho(A)+\eps)^{-n} \|\tilde x_n-\Phi(x_\infty)\|
		\\
		& \le \limsup_{n \to \infty} 2 (\rho(A)+\eps)^{-n} \sum_{i=n}^\infty\| \tilde x_{i+1}-\tilde x_{i} \| =0.
	\end{align*}
\end{proof}

\subsection{Local attraction} Next, we prove that, when initialized sufficiently close to $x^\ast$, $(x_n)_{n \in \N_0}$ converges to a local minimum in $\mathcal U$, where $\mathcal U$ denotes the set that is given by Lemma~\ref{lem:chart}. It is important to emphasize that, in contrast to the statement of Lemma~\ref{lem:conv}, the following result does not require the additional assumption that the sequence $(x_n)_{n \in \N_0}$ remains in a neighborhood of $x^\ast$. Instead, we show that the iterates are asymptotically attracted to the manifold $\mathcal M$ and slow down sufficiently fast to prevent them from escaping this neighborhood.

\begin{lemma} \label{lem:conv2}
	Let $\beta \in (0,1)$ and $\gamma \in \bigl(\frac{2(1+\beta)}{L}\bigr)$.
	For a $(\mu,L)$-regular point $x^\ast$ there exists a neighborhood $V$ of $x^\ast$ such that if $x_0,x_1 \in V$, then $(x_n)_{n\in\N_0}$ converges to a local minimum $x_\infty \in \mathcal U$, where $\mathcal U$ denotes the set that is given by Lemma~\ref{lem:chart}.
\end{lemma}

\begin{proof}
	Since $\rho(A)\le m(\gamma, \beta)<1$, see Lemma~\ref{lem:ev}, we can use \eqref{eq:23372987r9297} together with Lemma~\ref{lem:aux} to conclude that there exist a $\delta>0$ and a vanishing sequence $(C(n))_{n \in \N_0}$ such that for all $0 \le N \le n$
	\begin{align} \label{eq:6446}
		(x_i)_{0 \le i \le n} \subset \mathcal U_\delta \Rightarrow \|x_N-x_n\| \le C(N)\,.
	\end{align} 
	Choose $r_1>0$ with $B_{r_1}(x^\ast) \subset \mathcal U_\delta$. Moreover, choose 
	$r_2\in(0,r_1)$ such that $\|\nabla f(x)\| \le \frac{r_1}{8\gamma}$ for all $x \in B_{r_2}(x^\ast)$ and choose sufficiently large $N \in \N_0$ such that $C(n)< \frac{r_2}{8}$ for all $n \ge N$. 
	For all $n \in \N_0$, one has
	\begin{align*}
		\|x_{n+1}-x_n\| \le \gamma \|\nabla f(x_n)\| +\beta \|x_n-x_{n-1}\|
	\end{align*}
	and, using the Lipschitz continuity of $\nabla f$,
	\begin{align*}
		\|\nabla f(x_{n+1})\| &\le \|\nabla f(x_n)\| + L \|x_{n+1}-x_n\| \\
		&\le \|\nabla f(x_n)\| + L \gamma \|\nabla f(x_n)\| + L \beta \|x_n-x_{n-1}\|\,,
	\end{align*}
	which gives 
	\begin{align*}
		\|x_{n+1}-x_n\|+\|\nabla f(x_{n+1})\| \le (1+\gamma+L\gamma) \|\nabla f(x_n)\| + \beta(1+L) \|x_n-x_{n-1}\|\,.
	\end{align*}
	Repeating this inequality, one gets for $p=\max(1+\gamma+L\gamma, \beta(1+L))$ that
	\begin{align*}
		\|x_{n+1}-x_n\| + \|\nabla f(x_n)\| \le p^{n} (\|x_{1}-x_0\| + \|\nabla f(x_1)\|)\,.
	\end{align*}
	Thus, there exists
	$r\in(0,\sfrac{r_2}{4})$ such that $x_0,x_{1} \in B_r(x^\ast)=: V$ ensures 
	\begin{align*}
		\|x_{n+1}-x_n\| + \|\nabla f(x_n)\| \le \frac{r_2}{4 N} \quad \text{ for all }  n \le N \,,
	\end{align*}
	implying
	\begin{align} \begin{split} \label{eq:777}
			\|x_n-x^\ast\| &\le \|x_1-x^\ast\| + \sum_{i=2}^n \|x_i-x_{i-1}\| \\
			&\le \|x_1-x^\ast\| + \frac{r_2}{4} \le \frac{r_2}{2} \quad \text{ for all }  n \le N\,.
		\end{split}
	\end{align}
	By induction, we show that $x_0,x_1 \in B_r(x^\ast)$ implies
	\begin{align*}
		x_n \in B_{\frac{3}{4} r_2}(x^\ast) \subset \mathcal U_\delta \quad \text{ for all } n \in \N_0\,.
	\end{align*}
	For all $n \le  N$ this statement holds due to \eqref{eq:777}. Assume that the statement holds for all $n=0, \dots, N_1$ with $N_1 \ge N$. Then
	\begin{align*}
		\|x_{N_1+1}-x^\ast\| &\le \|x_{N_1+1}-x_{N_1}\| + \|x_{N_1} -x^\ast\| \le \|x_{N_1+1}-x_{N_1}\| + \frac{3}{4} r_2 \\
		& \le \gamma \|\nabla f(x_{N_1})\| + \beta \|x_{N_1}-x_{N_1-1}\| + \frac{3}{4} r_2
	\end{align*}
	Since $(x_n)_{n=0, \dots, N_1} \subset B_{r_2}(x^\ast)$ one has $\|\nabla f(x_{N_1})\| \le \frac{r_1}{8\gamma}$ and $\|x_{N_1}-x_{N_1-1}\|\le C(N_1) < \frac{r_2}{8}$ so that $\|x_{N_1+1}-x^\ast\|\le r_1$. Thus, $x_{N_1+1} \in B_{r_1}(x^\ast)\subset \mathcal U_\delta$ and, by \eqref{eq:6446}, one has $\|x_{N_1+1}-x_N\|\le \frac{r_2}{8}$. Together with $x_N \in B_{\frac 12 r_2}(x^\ast)$ we found $x_{n+1} \in B_{\frac 34 r_2}(x^\ast).$
	
	Finally, we are ready to apply Lemma~\ref{lem:conv}, to prove convergence of $(x_n)_{n \in \N_0}$ 
	towards a local minimum $x_\infty \in \mathcal U$. 
\end{proof}

\subsection{Final step: Local accelerated convergence}
We are now ready to prove the main result stated in Theorem~\ref{theo2}.
\begin{proof}[Proof of Theorem~\ref{theo2}]
	By Lemma~\ref{lem:conv2}, there exists a neighborhood $V$ of $x^\ast$ such that $x_0,x_1 \in V$ implies convergence of $(x_n)_{n \in \N_0}$ to a local minimum $x_\infty \in \mathcal U$, where $\mathcal U$ denotes the set that is given by Lemma~\ref{lem:chart}. Now, we can replace $x^\ast$ by $x_\infty$ in Lemma~\ref{lem:conv} and get the optimal rate of convergence. 
	In particular, for $\delta >0$ we denote by $\mathcal U_\delta(x_\infty)$ a neighborhood of $x_\infty$ satisfying the conclusion of Lemma~\ref{lem:conv} for $x^\ast$ replaced by $x_\infty$, which is also a $(\mu,L)$-regular point. Since $x_n \to x_\infty$, for every $\delta>0$ there exists an $N_\delta>0$ with $x_n \in \mathcal U_\delta(x_\infty)$ for all $n \ge N_\delta$. Thus, Lemma~\ref{lem:conv} implies that for all $\eps >0$
	\begin{align} \label{eq:as1}
		\lim_{n \to \infty} (m(\gamma, \beta)+\eps)^{-n}  \|x_n-x_\infty\| =0,
	\end{align}
	which implies the first assertion of Theorem~\ref{theo2}.
	
	For the second assertion note that, analogously to \eqref{eq:993725}, \eqref{eq:as1} immediately implies that
	\begin{align*}
		\limsup_{n \to \infty} \, (m(\gamma,\beta)+\eps)^{-2n} f(x_n)  =0.
	\end{align*}
\end{proof}

\section{Numerical experiment} \label{sec:experiments}
	In the following section, we illustrate our theoretical findings through a numerical toy example. This example shows that Polyak’s heavy ball method in discrete time exhibits local acceleration, with a local convergence rate that matches the rate predicted by 
	Theorem~\ref{theo2}. Moreover, it highlights that the convergence behavior is governed by the PL constant in the neighborhood of the limit point.
	
	Consider the function $g:\R\to\R$ given by $g(z) = z^2 + 3 \sin^2(z)$, which is non-convex but satisfies a (global) PL-inequality and has a unique global minimum at $z^\ast=0$; see also \cite{karimi2016linear}. 
	The global PL constant is given by
	$$
	\mu_{\mathrm{global}}:= \inf_{x \in \R} \mu(x) \approx 0.176 ,
	$$
	where $\mu(x) := \frac{1}{2}\frac{g'(x)^2}{g(x)}$. Since $\lim_{z\to 0} \mu(z) = g''(0) = 8$, a much stronger PL-inequality is satisfied around the global minimum $0$. In fact, by continuity, for any $\varepsilon>0$ there exists $\delta>0$ such that $g$ satisfies a PL-inequality with constant $8-\varepsilon$ on $(-\delta,\delta)$. 
	
	In our numerical experiment, we consider the heavy ball method \eqref{eq:HBdiscreteintro} applied to the two-dimensional objective function $f(x) = g(\psi(x_1,x_2))$ where $\psi(x_1,x_2):=x_2-\phi(x_1)$, $\phi(x_1):=0.7\, \sin(x_1)$, $x=(x_1,x_2)\in\R^2$. The set of minima is described by the graph of $\phi$, i.e.~$\cM=\{(x,\phi(x)): x\in\R\}\subset\R^2$. The gradient of $f$ satisfies
	\[ \|\nabla f(x)\|^2 = (1+(\phi')^2(x_1)) |g'(\psi(x_1,x_2))|^2\,.\]
	Since $g$ satisfies the PL-inequality with constant $\mu_{\mathrm{global}}$, we deduce a PL condition for $f$
	\[
	\|\nabla f(x)\|^2 \ge 2\mu_{\mathrm{global}}(1+ (\phi')^2(x_1)) f(x) \ge 2 \mu_{\mathrm{global}}f(x) 
	\]
	for all $x=(x_1,x_2)\in\R^2$. Moreover, for any $\varepsilon>0$ there exists $\delta>0$ such that for all $x \in V_\delta := \{(x_1,x_2)\in\R^2:|\psi(x_1,x_2)|=|x_2-\phi(x_1)|<\delta,\ x_1\in\R\}$ one has
	\[\frac12 \frac{\|\nabla f(x)\|^2 }{f(x)} \ge \frac{1}{2}\frac{(g')^2(\psi(x_1,x_2))}{g(\psi(x_1,x_2))} \ge 8-\varepsilon\,.\] 
	More precisely, there exists an open neighborhood $V_\delta$ of $\cM$ such that the PL-inequality holds with constant $\hat\mu_\delta = 8-\varepsilon$. Regarding the smoothness of $f$ we compute the Hessian as
	\[ \nabla^2 f(x) = g''(\psi(x_1,x_2)) \begin{pmatrix} -\phi'(x_1) \\ 1 \end{pmatrix} \begin{pmatrix} -\phi'(x_1) & 1 \end{pmatrix} + g'(\psi(x_1,x_2)) \begin{pmatrix}
		-\phi''(x_1) & 0 \\ 0 & 0
	\end{pmatrix}\]
	giving the bound
	\[
	\|\nabla^2 f(x)\| \le 8(0.7^2 + 1) + 0.7\, |g'(\psi(x_1,x_2)|,
	\]
	where we have used that $|g''(y)|\le 8$ for all $y \in \R$.
	In our implementation, following Theorem~\ref{theo2} we choose 
	$$
	\gamma=\frac{4}{\big(\sqrt{\hat \mu} + \sqrt{\hat L}\big)^2} \quad \text{ and } \quad \beta=\Big(\frac{\sqrt{\kappa}-1}{\sqrt{\kappa}+1}\Big)^2
	$$
	for different $\hat\mu\in\{\frac1{32}, 1, 4, 7.5\}$ including a global PL constant $\frac 1{32}< \mu_{\mathrm{global}}$ and approaching the local PL constant at $\mathcal M$. We initialize the method in $x_1(0) = 10$, $x_2(0) = \phi(x_1) + R$ for $R=5$ such that $|\psi(x_1(0),x_2(0))|\le R$. Provided that the iterates remain in $\{(x_1,x_2)\in \R^2 : |\psi(x_1,x_2)|\le 2R\}$ we can estimate the local smoothness constant $\hat L \le 8(0.7^2 + 1) + 2\cdot 0.7\cdot R$ which is used for our implementation.
	
	\begin{figure}[!htb]
		\centering \includegraphics[width=0.5\textwidth]{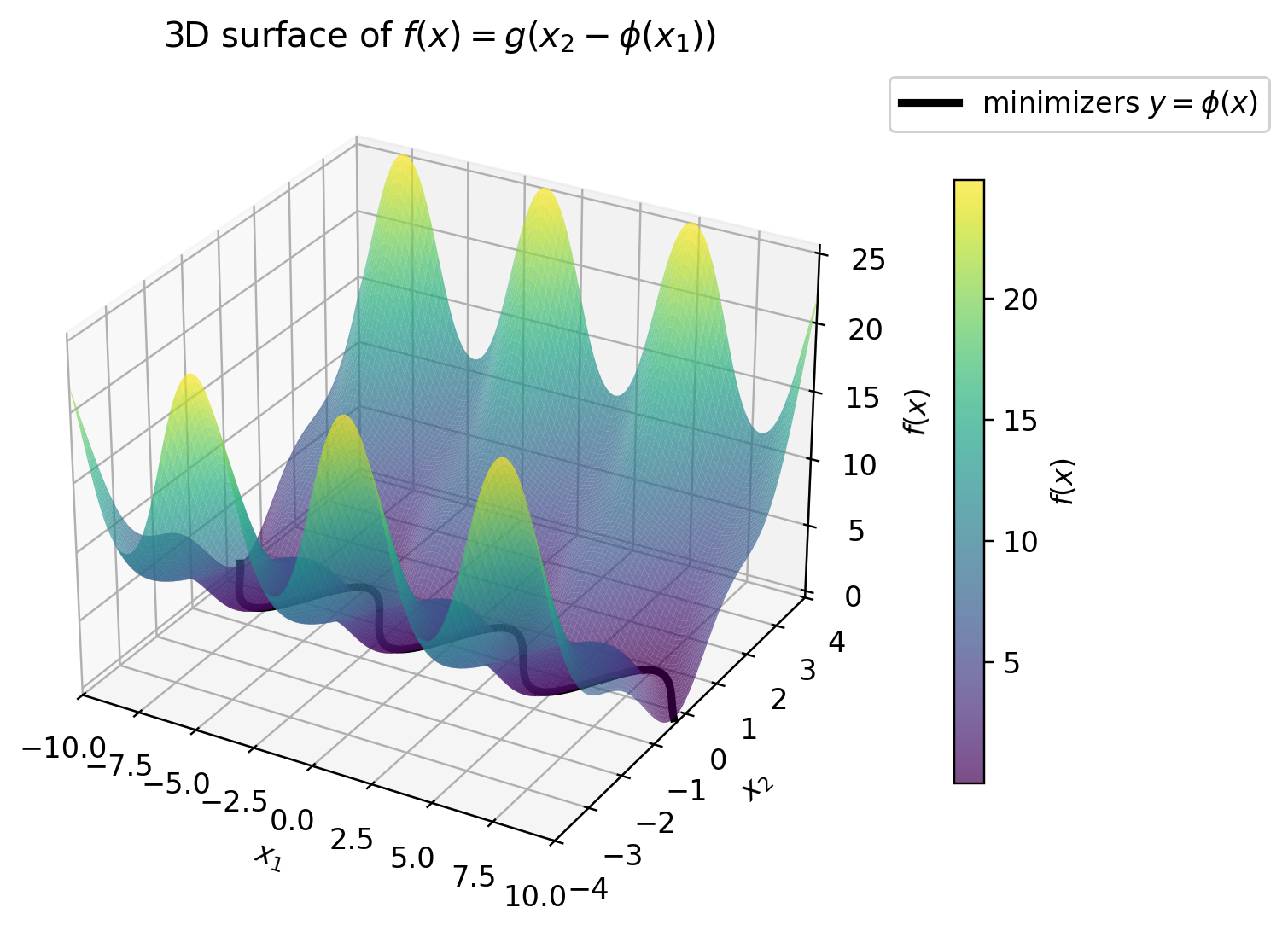}
		\caption{A three-dimensional surface plot of the objective function, illustrating the geometry of the optimization landscape. The set of global minimizers is visualized as black solid line. }  \label{fig:losslandscape}
	\end{figure} 
	
	\begin{figure}[!htb]
		\centering \includegraphics[width=0.45\textwidth]{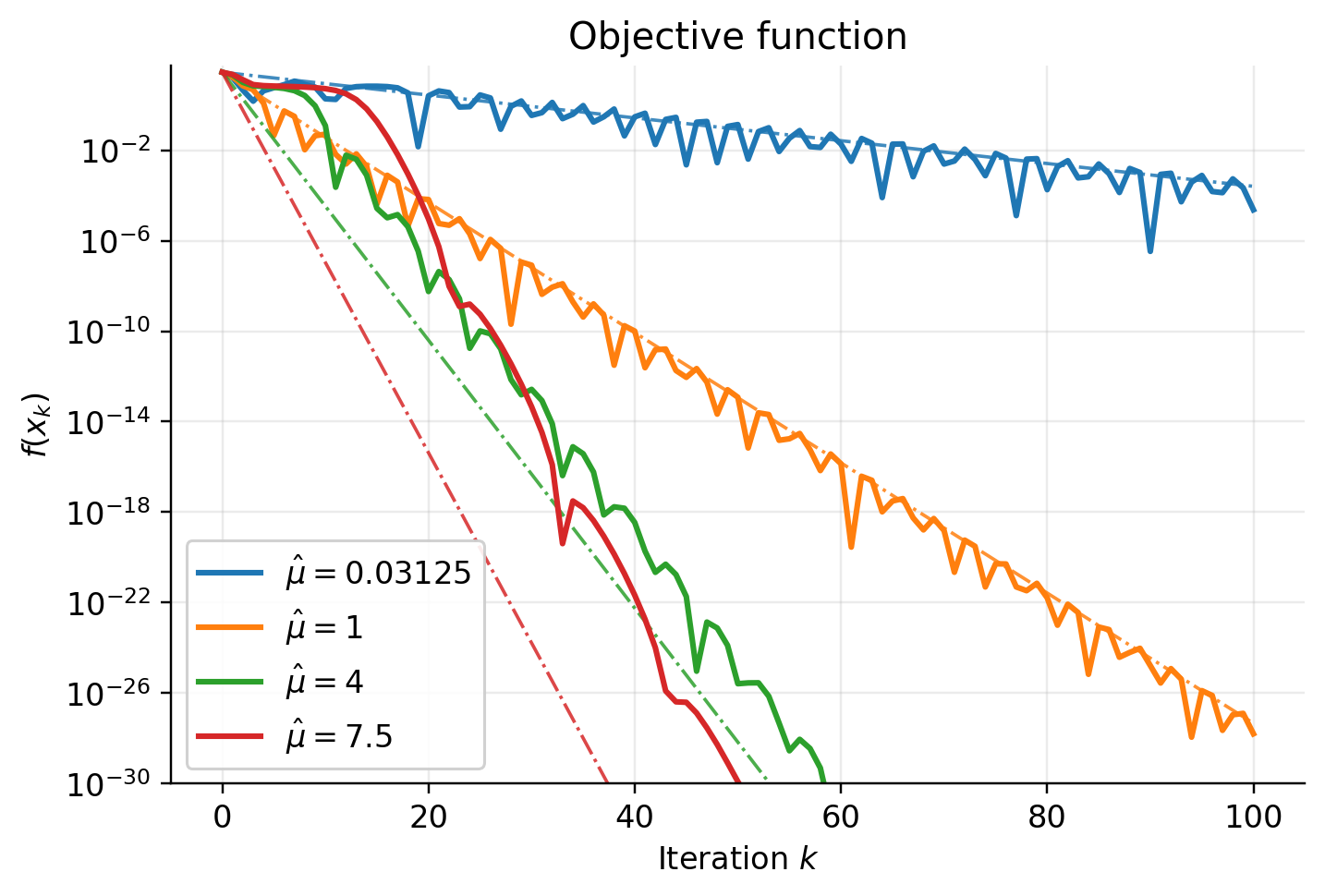}~~\includegraphics[width=0.45\textwidth]{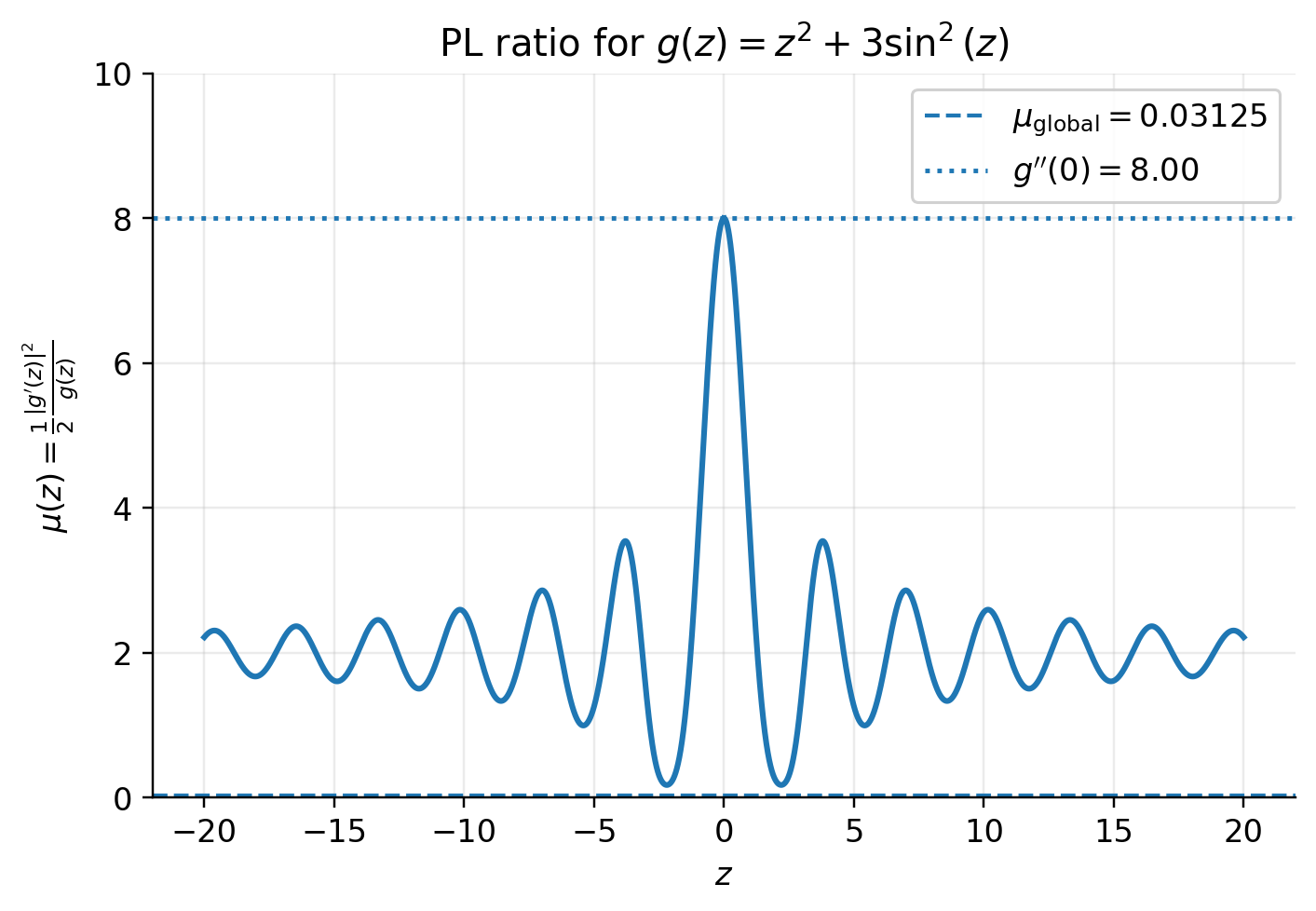}
		\caption{(Left) Evolution of the objective function values along the iterates for four different choices of estimated PL constants. Solid lines (in different colors) show the empirical losses, while the corresponding dashed lines represent the theoretical convergence rates predicted by our analysis, using the same color coding. (Right) Values of the local PL ratio $\mu(z):=\frac{1}{2}\frac{g'(z)^2}{g(z)}$ evaluated along the real line, illustrating how the effective PL constant varies with $z$.}  \label{fig:objective}
	\end{figure} 
	
	The numerical results illustrate the qualitative behavior predicted by our theoretical analysis, particularly the trade-off between asymptotic acceleration and transient behavior. The plot in Figure~\ref{fig:objective} (left) indicates that parameter choices closer to the maximal values allowed by the local PL constants results in a faster asymptotic convergence rate, consistent with the locally accelerated rate derived in Theorem~\ref{theo2}. At the same time, these more aggressive choices are associated with a longer burn-in phase before the accelerated regime becomes apparent. During this initial phase, the iterates have not yet entered the neighborhood in which the improved local constants govern the dynamics, and the observed convergence is correspondingly slower. Once this region is reached, the convergence behavior transitions to the predicted asymptotic regime.

		\bibliographystyle{alpha}
		\bibliography{Heavy_Ball_ODE_2}
		
	\end{document}